\documentclass[a4paper,twoside,reqno]{amsart}
\usepackage[utf8]{inputenc}

\usepackage{amsmath,amssymb,amsfonts,amsthm,enumerate,mathtools,tikz}

\usepackage{fullpage}
\usepackage{hyperref}
\renewcommand{\MR}[1]{\href{http://www.ams.org/mathscinet-getitem?mr=#1}{MR#1}}

\numberwithin{equation}{section}

\newcommand{\Span}{\mathrm{Span}}
\newcommand{\Nor}{\mathrm{Nor}}

\newcommand{\cA}{\mathcal{A}}

\newcommand{\cN}{\mathcal{N}}
\newcommand{\cT}{\mathcal{T}}
\newcommand{\cO}{\mathcal{O}}
\newcommand{\cU}{\mathcal{U}}

\newcommand{\cR}{\mathcal{R}}

\newcommand{\Lcal}{{\mathcal{L}}}
\newcommand{\Rcal}{{\mathcal{R}}}

\newcommand{\asX}{\langle X \rangle}
\newcommand{\asoX}{\langle X \rangle}
\newcommand{\asY}{\langle Y_N \rangle}

\newcommand{\la}{{\triangleright}}
\newcommand{\ra}{{\triangleleft}}

\newcommand{\N}{\mathbb{N}}

\newcommand{\Fcal}{{\mathcal{F}}}
\newcommand{\Dcal}{{\mathcal{D}}}
\newcommand{\Ocal}{{\mathcal{O}}}

\newcommand{\Sym}{\mathop{{\rm Sym}}\nolimits}

\def\id{\text{id}}
\theoremstyle{plain}

\newtheorem{thm}{Theorem}[section]
\newtheorem{pro}[thm]{Proposition}
\newtheorem{lem}[thm]{Lemma}
\newtheorem{cor}[thm]{Corollary}

\theoremstyle{definition}

\newtheorem{ex}[thm]{Example}

\newtheorem{fact}[thm]{Fact}

\newtheorem{defnotation}[thm]{Definition-Notation}
\newtheorem{convention}[thm]{Convention}

\newtheorem{dfn}[thm]{Definition}

\newtheorem{rmk}[thm]{Remark}
\newtheorem{problem}[thm]{Problems}

\newcommand{\LM}{\mathbf{LM}}

\title[Veronese subalgebras of Yang-Baxter algebras]{Veronese subalgebras and
Veronese morphisms for a class of  Yang-Baxter algebras}

\author{Tatiana Gateva-Ivanova}
\address{Max Planck Institute for Mathematics, Vivatsgasse 7, 53111 Bonn, Germany,
and American University in
Bulgaria, 2700 Blagoevgrad, Bulgaria} \email{tatyana@aubg.edu}

\thanks{Partially supported by the Max Planck Institute for Mathematics, (MPIM),
Bonn,
by the Abdus Salam International Centre for Theoretical Physics (ICTP), Trieste, and
by Grant KP-06 N 32/1, 07.12.2019 of the Bulgarian National Science Fund.
}

\begin{document}
\date{\today}

\begin{abstract}
We study $d$-Veronese subalgebras $\cA^{(d)}$ of Yang-Baxter algebras $\cA_X=
\cA(\textbf{k}, X, r)$ related to finite nondegenerate involutive set-theoretic
solutions $(X, r)$ of the Yang-Baxter equation, where $\textbf{k}$ is a
field and $d\geq 2$ is an integer. We find an explicit presentation of the $d$-Veronese $\cA^{(d)}$ in
terms of one-generators and quadratic relations. We introduce the notion of
\emph{a $d$-Veronese solution $(Y, r_Y)$}, canonically associated to $(X,r)$
and use its Yang-Baxter algebra $\cA_Y= \cA(\textbf{k}, Y, r_Y)$  to define a
Veronese morphism $v_{n,d}:\cA_Y  \rightarrow \cA_X $. We prove that the image
of $v_{n,d}$  is  the $d$-Veronese subalgebra $\cA^{(d)}$, and find explicitly a
minimal set of generators for its kernel.
The results agree with their classical analogues in the commutative case.
We show that the Yang-Baxter algebra $\cA(\textbf{k}, X, r)$ is a PBW algebra if and only if $(X,r)$ is a square-free solution. In this case the $d$-Veronese $A^{(d)}$ is also a PBW algebra.
   \end{abstract}
\maketitle

\section{Introduction}
\label{Intro}

 It was established in the last three decades that
solutions of the linear braid or Yang-Baxter equation (YBE)
\[r^{12}r^{23}r^{12}=r^{23}r^{12} r^{23} \]
on
a vector space of the form $V^{\otimes 3}$ lead to remarkable
algebraic structures. Here $r : V\otimes V \longrightarrow
V\otimes V,$  $r^{12}= r\otimes \id$, $r^{23} = id\otimes r$ is a
linear authomorphism and structures include coquasitriangular bialgebras
$A(r)$, their quantum group (Hopf algebra) quotients, quantum
planes and associated objects, at least in the case of specific
standard solutions, see \cite{MajidQG, RTF}. On the other hand,
the variety of all solutions on vector spaces of a given dimension
has remained rather elusive in any degree of generality. It was
proposed by V.G. Drinfeld \cite{D}, to consider the same equations
in the category of sets, and in this setting numerous  results
were found. It is clear that a set-theoretic solution extends to a
linear one, but more important than this is that set-theoretic
solutions lead to their own remarkable algebraic and combinatoric
structures, only somewhat analogous to quantum group
constructions. In the present paper we continue our systematic
study of set-theoretic solutions based on the associated quadratic
algebras and  monoids that they generate.

The study of non-commutative algebras defined by quadratic relations as examples of quantum non-commutative spaces has received considerable impetus from the seminal work of Faddev, Reshetikhin and Takhtadjan,\cite{FRT}, where the authors considered general deformations of quantum groups and spaces arising from an R-matrix, and from Manin's programme for non-commutative geometry \cite{Ma91}.  The quadratic algebras related to set-theoretic solutions of the Yang-Baxter equation studied here can be considered as special quantum non-commutative spaces important for both noncommutative algebra and non-commutative algebraic geometry, as they provide a rich source of examples of interesting associative algebras and non-commutative spaces some of which are Artin-Schelter regular algebras.
Our work is motivated by the relevance of those algebras for non-commutative geometry,  especially in relation to the theory of quantum groups, and inspired by the interpretation of morphisms between non-commutative algebras as "maps between non-commutative spaces". In
\cite{GI23} and the present paper we consider non-commutative analogues of the Veronese and Segre embeddings, two fundamental maps that play pivotal roles not only in classical algebraic geometry but also in applications to other fields of mathematics.


In this paper "a solution of YBE" , or shortly, "a solution" means "a
nondegenerate involutive set-theoretic solution of YBE", see Definition
\ref{def:quadraticsets_All}.

The  Yang-Baxter algebras $\cA_X= \cA(\textbf{k}, X, r)$ related to solutions $(X,
r)$ of finite order $n$
 will play a central role in the paper.
 It was proven in \cite{GIVB} and \cite{GI12}
that the  quadratic algebra $\cA_X$ of every finite solution $(X,r)$ of YBE has remarkable algebraic, homological and
combinatorial properties. In general, the algebra $\cA_X$ is noncommutative and in most cases it is not even a PBW algebra, but it preserves various good
properties of the commutative polynomial ring $\textbf{k} [x_1, \cdots , x_n]$:
$\cA_X$ has finite global dimension and polynomial growth, it is Cohen-Macaulay,
Koszul, and a Noetherian domain.


There are close relations between various combinatorial properties of the solution $(X,r)$ and the properties of the corresponding Yang-Baxter algebra $\cA_X$, see for example \cite{GI04, GI04s, rump, GIM11, GI12, GI21, vendramin}.
In the special case when $(X,r)$ is \emph{a finite nondegenerate involutive square-free} \emph{quadratic set}
whose quadratic algebra
$\cA_X= \cA(\textbf{k}, X, r)$ has a $\textbf{k}$-basis of Poincar\'{e}-Birkhoff-Witt type,  the conditions "\emph{$\cA$ is an Artin-Schelter regular algebra"} and "\emph{$(X,r)$ is a solution of YBE}"
\emph{are equivalent}, see details in Section 2.
The study of Artin-Schelter regular algebras is a central problem for noncommutative algebraic geometry.


A first stage of noncommutative geometry on quadratic algebras $\cA_X= \cA(\textbf{k}, X, r)$  was
proposed in  \cite{GIM11}, Section 6, where the quantum spaces under investigation are
Yang-Baxter algebras $\cA(\textbf{k}, X, r)$  associated to  multipermutation
(square-free) solutions of
level two.
In \cite{AGG} a class of quadratic PBW
algebras called "noncommutative projective spaces" were investigated
and analogues of  Veronese and Segre morphisms between noncommutative projective spaces were
introduced and studied. It is natural to formulate similar problems for the class of Yang-Baxter algebras $\cA= \cA(\textbf{k}, X, r)$ related to finite solutions $(X,r)$, but to find reasonable solutions of these problems is a nontrivial task. In contrast with \cite{AGG}, where the "noncommutative projective spaces" under investigation have almost commutative quadratic relations which form Gr\"{o}bner bases, and the main results follow naturally from the theory of Noncommutative Gr\"{o}bner bases, the Yang-Baxter algebras $\cA= \cA(\textbf{k}, X, r)$ have complicated quadratic relations, which in most cases do not form Gr\"{o}bner bases. These relations remain complicated even when $\cA$ is a PBW algebra, so we need more sophisticated arguments and techniques, see for example \cite{GI23}.

In the present paper we consider the following problems.
\begin{problem}
\label{problem}
\begin{enumerate}
Suppose $(X,r)$ is a finite solution of YBE with $|X|= n$, and
$\cA=\cA(\textbf{k}, X, r)$ is its Yang-Baxter algebra.
\item  Find necessary and sufficient conditions on $(X,r)$ such that there exists an enumeration $X=\{x_1, \cdots, x_n\}$, so that $\cA$ is a
PBW algebra with a set of PBW generators  $x_1, \cdots, x_n$.
\item  Let $d\geq 2$ be an integer.
Find a presentation of the $d$-Veronese subalgebra $\cA^{(d)}$ of its Yang-Baxter algebra
$\cA$ in terms of one-generators and quadratic relations.
\item Introduce analogues of Veronese maps for the class
of Yang-Baxter algebras of finite solutions of YBE.
\item Anser questions (2) and (3) in
 the special case when $(X,r)$ is a square-free solution.
\end{enumerate}
\end{problem}
Our main results are Theorem \ref{thm:PBWmain},
Theorem  \ref{thm:d-Veronese_relations}, Theorem
\ref{thm:Veronese_ker}.
which solve completely problems (1), (2), and (3). We give a complete answer to (4) in Section \ref{sec:square-free}.

The paper is organized as follows. In Section 2 we recall basic definitions and
facts used throughout the paper.
In Section \ref{sec:YBalgebras} we consider the Yang-Baxter algebra
$\cA_X= \cA(\textbf{k},  X, r)$ of a finite nondegenerate
solution $(X,r)$.
We fix the main settings and conventions and recall some of the most important
properties of the Yang-Baxter algebras
$\cA_X$  used throughout the paper.
The main result of the section is Theorem \ref{thm:PBWmain}  which shows that the Yang-Baxter algebra $\cA(\textbf{k}, X, r)$ is PBW with respect to some proper enumeration of $X$ \emph{iff}
the solution $(X,r)$ is square-free. Proposition  \ref{thm:PBW} gives more information on a special case of PBW quadratic algebras.
In Section \ref{sec:dVeronese} we study the $d$-Veronese subalgebra $\cA^{(d)}$ of $\cA=
\cA(\textbf{k}, X, r)$.
We use the fact that the algebra $\cA$ and its Veronese subalgebras are intimately
connected with the braided monoid $S(X,r)$. To solve the main  problem we
introduce successively  three finite isomorphic solutions associated naturally to
$(X,r)$,  and
involved in the proof of our results.
The first and the most natural of the three is \emph{the monomial
$d$-Veronese solution $(S_d, r_d)$ associated with} $(X,r)$. It is a finite solution induced from
the graded braided monoid $(S, r_S)$ which depends only on the map $r$ and on $d$.
The monomial $d$-Veronese solution  is intimately connected with the $d$-Veronese subalgebra
$\cA^{(d)}$ and its quadratic relations, but it is not convenient for an explicit
description of the relations. This solution is needed to
 define  \emph{the normalised $d$-Veronese solution $(\cN_d, \rho_d)$} isomorphic to $(S_d, r_d)$,
see Definition \ref{def:normalizedSol}. The solution $(\cN_d, \rho_d)$ is central for the proof of the main result
Theorem  \ref{thm:d-Veronese_relations}.
In Section \ref{sec:Veronesemap}
we introduce and study analogues of Veronese maps between Yang-Baxter algebras of
finite solutions and prove Theorem
\ref{thm:Veronese_ker}.
In Section \ref{sec:square-free} we consider two special cases of solutions. We pay special attention to
Yang-Baxter algebras $\cA= \cA(\textbf{k}, X, r)$
of square-free solutions $(X,r)$ and their Veronese subalgebras. In this case
$\cA$ is a binomial skew polynomial ring and  the set of ordered monomials (terms) in $n$ variables forms an explicit $\textbf{k}$- basis of $\cA$.
Theorem \ref{thm:d-Veronese_relations} implies a more precise result in this case:
Corollary \ref{cor:d-Veronese_relations_square-free}
shows that the $d$-Veronese $A^{(d)}$ is a PBW algebra, where the terms of length $d$ ordered lexicographically are its PBW generators and its  relations given explicitly form a quadratic Gr\"{o}bner basis.
An important result in this section is
Theorem \ref{thm:Sd_square-free} which shows that if $(X,r)$ is a finite
square-free solution and $d \geq 2$ is an integer, then the monomial $d$-Veronese
solution
$(S_d, r_d)$  is
square-free if and only if $(X,r)$ is a trivial solution.
This implies that the notion of Veronese morphisms introduced for the class of
Yang-Baxter algebras of finite solutions
can not be restricted to the subclass of Yang-Baxter algebras associated to finite square-free
solutions.
In Section \ref{sec:examples} we present two examples which illustrate the
results of the paper.
\section{Preliminaries}
\label{seq:preliminaries}
Let $X$ be a non-empty set, and let $\textbf{k}$ be a field.
We denote by $\asX$
the free monoid
generated by $X,$ where the unit is the empty word denoted by $1$, and by
$\textbf{k}\asX$-the unital free associative $\textbf{k}$-algebra
generated by $X$. For a non-empty set $F
\subseteq \textbf{k}\asX$, $(F)$ denotes the two sided ideal
of $\textbf{k}\asX$ generated by $F$.
When the set $X$ is finite, with $|X|=n$,  and ordered, we write $X= \{x_1,
\dots, x_n\},$ and fix the degree-lexicographic order $<$ on $\asX$, where $x_1< \dots
<x_n$.
As usual, $\N$ denotes the set of all positive integers,
and $\N_0$ is the set of all non-negative integers.

We shall consider associative graded $\textbf{k}$-algebras.
Suppose $A= \bigoplus_{m\in\N_0}  A_m$ is a graded $\textbf{k}$-algebra
such that $A_0 =\textbf{k}$, $A_pA_q \subseteq A_{p+q}, p, q \in \N_0$, and such
that $A$ is finitely generated by elements of positive degree. Recall that its
Hilbert function is $h_A(m)=\dim A_m$
and its Hilbert series is the formal series $H_A(t) =\sum_{m\in\N_0}h_{A}(m)t^m$.
In particular, the algebra $\textbf{k} [X]$ of commutative polynomials satisfies
\begin{equation}
\label{eq:hilbert}
h_{\textbf{k} [X]}(d)= \binom{n+d-1}{d}= \binom{n+d-1}{n-1} \quad\mbox{and}\quad
H_{\textbf{k} [X]}= \frac{1}{(1 -t)^{n}}.
\end{equation}
We shall use the \emph{natural grading by length} on the free associative algebra
$\textbf{k}\asX$.
For $m \geq 1$,  $X^m$ will denote the set of all words of length $m$ in $\asX$,
where the length of $u = x_{i_1}\cdots x_{i_m} \in X^m$
will be denoted by $|u|= m$.
Then
\[\asX = \bigsqcup_{m\in\N_0}  X^{m},\;
X^0 = \{1\},\;  \mbox{and} \ \   X^{k}X^{m} \subseteq X^{k+m},\]
so the free monoid $\asX$ is naturally \emph{graded by length}.

Similarly, the free associative algebra $\textbf{k}\asX$ is also graded by
length:
\[\textbf{k}\asX
 = \bigoplus_{m\in\N_0} \textbf{k}\asX_m,\quad \mbox{ where}\ \
 \textbf{k}\asX_m=\textbf{k} X^{m}. \]

A polynomial $f\in  \textbf{k}\asX$ is \emph{homogeneous of degree $m$} if $f \in
\textbf{k} X^{m}$.
We denote by
\begin{equation}
\label{eq:terms}
\cT =\cT(X) :=\left\lbrace x_1^{\alpha_1}\cdots x_n^{\alpha_n}\in\asX \ \vert
\ \alpha_i\in\N_0, i\in\{0,\dots,n\}\right\rbrace
\end{equation}
the set of ordered monomials (terms) in $\asX$   and by
\[
\cT_d = \cT(X)_d:=\left\lbrace x_1^{\alpha_1}\cdots x_n^{\alpha_n}\in\cT \mid\
\sum_{i=1}^n \alpha_i = d \right\rbrace
\]
the set of ordered monomials of length $d$.

\subsection{Gr\"obner bases for ideals in the free associative algebra}
\label{sec:grobner}
We shall remind some basics of noncommutative Gr\"obner bases theory which we use throughout in the paper. In this subsection $X=
\{x_1,\dotsc,x_n\}$, we fix the degree lexigographic order $<$ on the free monoid $\asX$ extending $x_1 < x_2< \cdots <x_n$ (we refer to it as "deg-lex ordering").
Suppose $f \in \textbf{ k}\asX$ is a nonzero polynomial. Its leading
monomial with respect to the deg-lex order $<$ will be denoted by
$\LM(f)$.
One has $\LM(f)= u$ if
$f = cu + \sum_{1 \leq i\leq m} c_i u_i$, where
$ c, c_i \in \textbf{k}$, $c \neq 0 $ and $u > u_i$ in $\asX$, for all
$i\in\{1,\dots,m\}$.
Given a set $F \subseteq \textbf{k} \asX$ of
non-commutative polynomials, we consider the set of leading monomials
 $\LM(F) = \{\LM(f) \mid f \in F\}.
 $
A monomial $u\in \asX$ is \emph{normal modulo $F$} if it does not contain any of
the monomials $\LM(f)$ as a subword.
 The set of all normal monomials modulo $F$ is denoted by $N(F)$.

Let  $I$ be a two sided graded ideal in $K \asX$ and let $I_m = I\cap
\textbf{k}X^m$.
We shall
assume that
$I$ \emph{is generated by homogeneous polynomials of degree $\geq 2$}
and $I = \bigoplus_{m\ge 2}I_m$. Then the quotient
algebra $A = \textbf{k} \asX/ I$ is finitely generated and inherits its
grading $A=\bigoplus_{m\in\N_0}A_m$ from $ \textbf{k} \asX$. We shall work with
the so-called \emph{normal} $\textbf{k}$-\emph{basis of} $A$.
We say that a monomial $u \in \asX$ is  \emph{normal modulo $I$} if it is normal
modulo $\LM(I)$. We set
$N(I):=N(\LM(I))$.
In particular, the free
monoid $\asX$ splits as a disjoint union
\begin{equation}
\label{eq:X1eq2a}
\asX=  N(I)\sqcup \LM(I).
\end{equation}
The free associative algebra $\textbf{k} \asX$ splits as a direct sum of
$\textbf{k}$-vector
  subspaces
  \[\textbf{k} \asX \simeq  \Span_{\textbf{k}} N(I)\oplus I,\]
and there is an isomorphism of vector spaces
$A \simeq \Span_{\textbf{k}} N(I).$

It follows that every $f \in \textbf{k}\asX$ can be written uniquely as $f =
h+f_0,$ where $h \in I$ and $f_0\in {\textbf{k}} N(I)$.
The element $f_0$ is called \emph{the normal form of $f$ (modulo $I$)} and denoted
by
$\Nor(f)$
We define
\[N(I)_{m}=\{u\in N(I)\mid u\mbox{ has length } m\}.\]
In particular, $N(I)_1 =X, N(I)_0=1$. Then
$A_m \simeq \Span_{\textbf{k}} N(I)_{m}$ for every $m\in\N_0$.

A subset
$G \subseteq I$
of monic polynomials is a \emph{Gr\"{o}bner
basis} of $I$ (with respect to the order $<$) if
\begin{enumerate}
\item $G$ generates $I$ as a
two-sided ideal, and
\item for every $f \in I$ there exists $g \in G$ such that $\LM(g)$ is a
    subword of $\LM(f)$, that is
$\LM(f) = a\LM(g)b$,  for some $a, b \in \asX$.
\end{enumerate}
A  Gr\"{o}bner basis $G$ of
$I$ is \emph{reduced} if (i)  the set $G\setminus\{f\}$ is not a Gr\"{o}bner
basis of $I$, whenever $f \in G$; (ii) each  $f \in G$  is a linear combination
of normal monomials modulo $G\setminus\{f\}$.

It is well-known that every ideal $I$ of $\textbf{k} \asX$ has a unique reduced
Gr\"{o}bner basis $G_0= G_0(I)$ with respect to $<$. However, $G_0$ may be
infinite.
 For more details, we refer the reader to \cite{Latyshev,Mo88, Mo94}.

The set of leading monomials  of the reduced Gr\"{o}bner basis $G_0= G_0(I)$
\begin{equation}
\label{eq:obstructions}
W =\{LM(f) \mid  f \in G_0(I)\}
\end{equation}
is \emph{the set of obstructions} for $A= \textbf{k} \asX/ I$, in the sense of Anick, \cite{Anickmon}.
There are equalities of sets  $N(I) = N(G_0) = N(W).$ We shall use the set of obstructions for the proof of Theorem \ref{thm:PBWmain}.

Bergman's Diamond lemma  \cite[Theorem 1.2]{Bergman} implies the following.
\begin{rmk}
\label{rmk:diamondlemma}
Let $G  \subset \textbf{k}\asX$  be a set  of noncommutative polynomials. Let $I =
(G)$ and let $A = \textbf{k}\asX/I.$ Then the following
conditions are equivalent.
\begin{enumerate}
\item
The set
$G$  is a Gr\"{o}bner basis of $I$.

\item Every element $f\in \textbf{k}\asX$ has a unique normal form modulo $G$,
    denoted by $\Nor_G (f)$.
\item
There is an equality $N(G) = N(I)$, so there is an isomorphism of vector spaces
\[\textbf{k}\asX \simeq I \oplus \textbf{k}N(G).\]
\item The image of $N(G)$ in $A$ is a $\textbf{k}$-basis of $A$.
In this case
$A$ can be identified with the $\textbf{k}$-vector space $\textbf{k}N(G)$, made
a $\textbf{k}$-algebra by the multiplication
$a\bullet b: = \Nor(ab).$
\end{enumerate}
\end{rmk}

 In this paper, we focus on a class of quadratic finitely presented algebras $A$
 associated with set-theoretic nondegenerate involutive solutions $(X,r)$ of
 finite order $n$. Following Yuri Manin, \cite{Ma88}, we call them \emph{Yang-Baxter algebras}.

\subsection{Quadratic algebras}
\label{sec:Quadraticalgebras}
A quadratic  algebra is an associative graded algebra
 $A=\bigoplus_{i\ge 0}A_i$ over a ground field
 $\textbf{k}$  determined by a vector space of generators $V = A_1$ and a
subspace of homogeneous quadratic relations $R= R(A) \subset V
\otimes V.$ We assume that $A$ is finitely generated, so $\dim A_1 <
\infty$. Thus $ A=T(V)/( R)$ inherits its grading from the tensor
algebra $T(V)$.

Following the classical tradition (and a recent trend), we take a
combinatorial approach to study $A$. The properties of $A$ will be
read off a presentation $A= \textbf{k} \langle X\rangle /(\Re)$,
where by convention $X$ is a fixed finite set of generators of
degree $1$, $|X|=n,$
and $(\Re)$ is the two-sided
ideal of relations, generated by a {\em finite} set $\Re$ of
homogeneous polynomials of degree two. In particular, $A_1 = V= \Span_{\textbf{k} } X$.
 \begin{dfn}
\label{def:PBW}
A quadratic algebra $A$ is
\emph{a  Poincar\`{e}–Birkhoff–Witt type algebra} or shortly
\emph{a PBW algebra} if there exists an enumeration $X=\{x_1,
\cdots, x_n\}$ of $X,$ such that the quadratic relations $\Re$ form a
(noncommutative) Gr\"{o}bner basis with respect to the
deg-lex order $<$ on $\asX$.
In this case the set of normal monomials
(mod $\Re$) forms a $\textbf{k}$-basis of $A$ called a \emph{PBW
basis}
 and $x_1,\cdots, x_n$ (taken exactly with this enumeration) are called \emph{
 PBW-generators of $A$}.
\end{dfn}
 \emph{PBW} algebras were introduced by Priddy, \cite{priddy}.
 The \emph{PBW basis}  is a generalization of the classical
 Poincar\'{e}-Birkhoff-Witt basis for the universal enveloping of a  finite
 dimensional Lie algebra.
 PBW algebras form an important class of Koszul algebras.
  The interested reader can find information on quadratic algebras and, in
  particular, on Koszul algebras and PBW algebras in
  \cite{PoPo}.
A special class of PBW algebras important for this paper, are the
\emph{binomial skew polynomial rings} introduced and studied by the author in \cite{GI94, GI96}.
  \begin{dfn}
\label{binomialringdef}
\cite{GI96, GI94} A {\em binomial skew polynomial ring} is a quadratic algebra
 $A=\textbf{k} \langle x_1, \cdots , x_n\rangle/(\Re_0)$ with
precisely $\binom{n}{2}$ defining relations
\begin{equation}
\label{eq:skewpol}
\Re_0=\{\varphi_{ji} = x_{j}x_{i} -
c_{ij}x_{i^\prime}x_{j^\prime} \mid 1\leq i<j\leq n\}
\end{equation}
 such that
(a) $c_{ij} \in \textbf{k}^{\times}$;
(b) For every pair $i, j, \; 1\leq
i<j\leq n$, the relation $x_{j}x_{i} - c_{ij}x_{i'}x_{j'}\in \Re_0,$
satisfies $j
> i^{\prime}$, $i^{\prime} < j^{\prime}$;
(c) Every ordered monomial $x_ix_j,$
with $1 \leq i < j \leq n$ occurs (as a second term) in some
relation in $\Re_0$;
(d) $\Re_0$ is the
{\it reduced Gr\"obner basis\/} of the two-sided ideal $(\Re_0)$,
with respect to the deg-lex order $<$ on $\asX$;   or
equivalently,
(d$^{\prime}$)  The set of terms $\cT=\left\lbrace x_1^{\alpha_1}\cdots x_n^{\alpha_n}\in\asoX \ \vert
\ \alpha_i\in\N_0, i\in\{0,\dots,n\}\right\rbrace$ projects to a $\textbf{k}$-basis of $A$.
\end{dfn}
The equivalence of (d) and (d$^{\prime}$)  follows from the Diamond Lemma, see Remark \ref{rmk:diamondlemma}.

It is clear that each binomial skew polynomial ring $A$ is a PBW algebra with a set of PBW generators $x_1, \cdots , x_n$.
It was proven  in \cite{GIVB} that $A$ defines via its relations a square-free
solution of the Yang-Baxter equation. Conversely, if $(X,r)$ is a finite square-free solution, then there exists an enumeration
$X=\{x_1, x_2, \cdots, x_n\}$ such that the Yang-Baxter algebra $\cA(\textbf{k}, X,r)$ is a binomial skew polynomial ring, this follows from results of Rump,  \cite{rump}, see also details in \cite{GI04}.

\begin{ex}
\label{example1}
Let $A = \textbf{k}\langle x_1,x_2,x_3,x_4 \rangle /(\Re_0)$, where
\[
 \begin{array}{l}
\Re_0
 = \{x_4x_2 -x_1x_3,\; x_4x_1 -x_2x_3,\;x_3x_2 -x_1x_4,\; x_3x_1 -x_2x_4, \;
 x_4x_3 -x_3x_4,\;  x_2x_1 -x_1x_2
\}.
\end{array}
\]
The algebra $A$ is a binomial skew polynomial ring. It is a PBW algebra with PBW
generators $X= \{x_1, x_2, x_3, x_4\}$.
The relations of $A$ define in a natural way a solution of YBE.
 \end{ex}

\subsection{Set-theoretic solutions of the Yang-Baxter equation and their Yang-Baxter algebras}
The notion of \emph{a quadratic set} was introduced in \cite{GI04}, see also \cite{GIM08},  as a
set-theoretic analogue of a quadratic algebra.
\begin{dfn}\cite{GI04}
\label{def:quadraticsets_All} Let $X$ be a nonempty set (possibly
infinite) and let $r: X\times X \longrightarrow X\times X$ be a
bijective map. In this case we use notation $(X, r)$ and refer to
it as \emph{a quadratic set}. The image of $(x,y)$ under $r$ is
written as
\[
r(x,y)=({}^xy,x^{y}).
\]
This formula defines a ``left action'' $\Lcal: X\times X
\longrightarrow X,$ and a ``right action'' $\Rcal: X\times X
\longrightarrow X,$ on $X$ as: $\Lcal_x(y)={}^xy$, $\Rcal_y(x)=
x^{y}$, for all $x, y \in X$. (i) $(X, r)$ is \emph{non-degenerate}, if
the maps $\Lcal_x$ and $\Rcal_x$ are bijective for each $x\in X$.
(ii) $(X, r)$ is \emph{involutive} if $r^2 = id_{X\times X}$. (iii)
$(X,r)$ is \emph{square-free} if $r(x,x)=(x,x)$ for all $x\in X.$
(iv) $(X, r)$ is \emph{a set-theoretic
solution of the Yang--Baxter equation} (YBE) if  the braid
relation
\[r^{12}r^{23}r^{12} = r^{23}r^{12}r^{23}\]
holds in $X\times X\times X,$  where  $r^{12} = r\times\id_X$, and
$r^{23}=\id_X\times r$. In this case we  refer to  $(X,r)$ also as
\emph{a braided set}. (v)  A braided set $(X,r)$ with $r$ involutive is
called \emph{a symmetric set}. (vi) A nondegenerate symmetric set is called simply \emph{a solution}.

$(X,r)$ is
\emph{the trivial solution} on $X$ if $r(x,y)=(y,x)$ for all $x, y \in X$.
\end{dfn}
\begin{rmk}
    \label{rmk:YBE1}
\cite{ESS}
Let $(X,r)$ be quadratic set.
Then $r$ obeys the YBE,
that is $(X,r)$ is a braided set {\em iff} the following three conditions
hold for all $x,y,z \in X$:
\[
\begin{array}{cccc}
 {\bf l1:}&{}^x{({}^yz)}={}^{{}^xy}{({}^{x^y}{z})},
 \quad\quad
 {\bf r1:}&
{(x^y)}^z=(x^{{}^yz})^{y^z},
 \quad \quad{\rm\bf lr3:}&
{({}^xy)}^{({}^{x^y}{z})} \ = \ {}^{(x^{{}^yz})}{(y^z)}.
\end{array}\]
The map $r$ is involutive {\em iff}
\[{\bf inv:}\quad {}^{{}^xy}{(x^y)} = x, \; \text{and}\; ({}^xy)^{x^y}=y.\]
\end{rmk}

\begin{convention}
\label{conv:convention1} In this paper
"\emph{A solution of YBE}", or simply "\emph{a solution}"
means "\emph{a non-degenerate symmetric set}" $(X,r)$, where $X$  is
a set of arbitrary cardinality.
\end{convention}
As a notational tool, we  shall  identify the sets $X^{\times
m}$ of ordered $m$-tuples, $m \geq 2,$  and $X^m,$ the set of all
monomials of length $m$ in the free monoid $\asX$. We shall use also notation $\cdot r(x,y):= xy$. Sometimes for simplicity we
 shall write $r(xy)$ instead of $r(x,y)$.

\begin{dfn} \cite{GI04, GIM08}
\label{def:algobjects} To each quadratic set $(X,r)$ we associate
canonically algebraic objects generated by $X$ and with quadratic
relations $\Re =\Re(r)$ naturally determined as
\[
xy=y^{\prime} x^{\prime}\in \Re(r)\; \text{iff}\;
 r(x,y) = (y^{\prime}, x^{\prime})\; \text{and} \;
 (x,y) \neq (y^{\prime}, x^{\prime})\;\text{hold in}\;X \times X.
\]
 The monoid
$S =S(X, r) = \langle X ; \; \Re(r) \rangle$
 with a set of generators $X$ and a set of defining relations $ \Re(r)$ is
called \emph{the monoid associated with $(X, r)$}.
 The \emph{group $G=G(X, r)= G_X$ associated with} $(X, r)$ is
defined analogously.
For an arbitrary fixed field $\textbf{k}$,
\emph{the} \textbf{k}-\emph{algebra associated with} $(X ,r)$ is
defined as
\[\begin{array}{c}
\cA = \cA(\textbf{k},X,r) = \textbf{k}\langle X  \rangle /(\Re_{\cA})
\simeq \textbf{k}\langle X ; \;\Re(r)
\rangle,\;\text{where}\;
\Re_{\cA} = \{xy-y^{\prime}x^{\prime}\mid xy=y^{\prime}x^{\prime}\in \Re
(r) \}.
\end{array}
\]
Clearly, $\cA$ is a quadratic algebra generated by $X$ and
 with defining relations  $\Re_{\cA}$, which is
isomorphic to the monoid algebra $\textbf{k}S(X, r)$.

When $(X,r)$ is a solution of YBE,  the
algebra $\cA$ is called \emph{an Yang-Baxter algebra}, \cite{Ma88},or shortly \emph{YB algebra}.
\end{dfn}

Suppose $(X,r)$ is a finite quadratic set. Then
$\cA=  \cA(\textbf{k},X,r)$ is \emph{a connected
graded} $\textbf{k}$-algebra (naturally graded by length),
 $\cA=\bigoplus_{i\ge0}\cA_i$, where
$\cA_0=\textbf{k},$
and each graded component $\cA_i$ is finite dimensional.
Moreover, the associated monoid $S= S(X,r)$ \emph{is naturally graded by
length}:
\begin{equation}
\label{eq:Sgraded}
S = \bigsqcup_{i \geq 0}  S_{i}, \;\; \text{where}\;\;
S_0 = 1,\; S_1 = X,\; S_i = \{u \in S \mid\;   |u|= i \}, \; S_i.S_j \subseteq
S_{i+j}.
\end{equation}
In the sequel, by "\emph{a graded monoid} $S$", we shall
mean that $S$ is generated  by $S_1=X$ and graded by length.
The grading of $S$ induces a canonical grading of its monoid algebra
$\textbf{k}S(X, r).$
The isomorphism $\cA\cong \textbf{k}S(X, r)$ agrees with the canonical gradings,
so there is an isomorphism of vector spaces $\cA_m \cong Span_{\textbf{k}}S_m$.

 If
$(X,r)$ is a nondegenerate  involutive quadratic set of finite order $|X| =n$
then, by \cite[Proposition 2.3.]{GI11},  the set $\Re$
consists of precisely $\binom{n}{2}$ quadratic relations.
In this  case the associated algebra $\cA= \cA(\textbf{k}, X, r)$ satisfies
$\dim \cA_2 = \binom{n+1}{2}.$

\begin{defnotation}\cite{GI21}
\label{def:fixedpts}
Suppose $(X,r)$ is an involutive quadratic set. Then the cyclic group $\langle r \rangle =\{1, r\}$  acts on the set $X^2$ and  splits it into disjoint $r$-orbits $\{ xy, r(xy)\}$, where $xy\in X^2$.
An $r$-orbit $\{ xy, r(xy)\}$
is \emph{non-trivial} if $xy\neq r(xy)$.
The element $xy\in X^2$ is \emph{an
$r$-fixed point} if $r(xy) =xy$. \emph{The set of $r$-fixed points} in $X^2$
will be denoted by $\Fcal (X,r)$:
\begin{equation}
\label{eq:fixedpts}
\Fcal (X,r) = \{xy\in X^2\mid r(xy) = xy\}.
\end{equation}
\end{defnotation}

The following useful corollary is a consequence from \cite[Lemma 3.7]{GI21}.
\begin{cor}
\label{cor:orbits_and_fixedpoints}
Let  $(X,r)$ be a solution of YBE of finite order $|X|= n$, and let $\cA= \cA(\textbf{k}, X, r)$ be its Yang-Baxter algebra.
(1) There are exactly $n$ fixed points
$\Fcal = \Fcal (X,r) =  \{x_1y_1, \cdots , x_ny_n\}\subset X^2$ ,
so
$|\Fcal(X,r)|= |X|= n.$
In the special case, when $(X,r)$ is a square-free solution, one has
$\Fcal(X,r) = \Delta_2=\{xx\mid x\in X\}$, the diagonal of $X^2$.
(2) The number of non-trivial $r$-orbits is exactly $\binom{n}{2}$. Each such an orbit has two distinct elements:
$xy$ and $r(xy)$, where $xy, r(xy) \in X^2$.
(3)
The set $X^2$  splits into $\binom{n+1}{2}$ $r$-orbits.
For $xy,zt\in X^2$ there is an equality $xy=zt$ in $\cA$ \emph{iff} $zt\in \{xy, r(xy)\}$. (4) In particular,  $\cA$ has exactly $\binom{n}{2}$ defining relations (each relation corresponds to a nontrivial $r$ orbit).
\end{cor}

\begin{rmk}
\label{orbitsinG}
 \cite{GI04s}
 Let $(X,r)$ be an involutive quadratic set, and let $S=S(X,r)$ be the
 associated monoid.

 (i) By definition, two monomials $w, w^{\prime} \in \langle X\rangle$ are equal
 in $S$
  \emph{iff} $w$ can be transformed to $w^{\prime}$ by a finite sequence of
  replacements each of the form
  \[axyb \longrightarrow ar(xy)b, \;
  \quad \text{where}\;x,y \in X,  a,b \in \langle X\rangle.\]

 Clearly, every such replacement preserves monomial length, which therefore
 descends to $S(X,r)$.
Furthermore, replacements coming from the defining relations are possible only
on monomials of length $\geq 2$, hence $X \subset S(X,r)$ is an inclusion. For
monomials of length $2$,
  $xy =zt$ holds in $S(X,r)$ \emph{iff} $zt = r(xy)$ is an equality of words in
  $X^2$.

(ii) It is  convenient for each $m \geq 2$ to refer to the subgroup $D_m=
D_m(r)$ of the symmetric group $\Sym (X^m)$
generated concretely by the maps
\begin{equation}
\label{eq:r{ii+1}}
r^{ii+1}: X^m \longrightarrow X^m, \; r^{ii+1} =\id_{X^{i-1}}\times r\times
\id_{X^{m-i-1}}, \; i = 1, \cdots, m-1.
\end{equation}
One can also consider the free groups
\[\Dcal_m (r)=  \: _{\rm{gr}} \langle r^{i i+1}\mid i = 1, \cdots, m-1
\rangle,\]
where the
$r^{ii+1}$ are treated as abstract symbols, as well as various quotients
depending on the further type of $r$ of interest.
These free groups and their quotients act on $X^m$ via the actual maps
$r^{ii+1}$, so that
 the image of $\Dcal_m(r)$ in $Sym (X^m)$ is $D_m(r).$ In particular, $D_2(r) =
 \langle r \rangle \subset Sym (X^2)$
 is the cyclic group generated by $r$.
 It follows straightforwardly from part (i) that $w, w^{\prime} \in \langle
 X\rangle$ are equal as words in $S(X,r)$
  \emph{iff} they have the same length, say $m$, and belong to the same orbit
  $\Ocal_{\Dcal_m}$ of $\Dcal_m(r)$ in $X^m$. In this case the equality
  $w=w^{\prime}$ holds in $S(X,r)$ and in the algebra $\cA(\textbf{k}, X, r)$.

An effective part of our combinatorial approach is the exploration of the
action of the group
$\Dcal_2(r) = \langle r\rangle$ on $X^2$,
and the properties of the corresponding orbits.  In the
literature a $\Dcal_2(r)$-orbit $\Ocal$ in $X^2$ is often called
"\emph{an $r$-orbit}" and we shall use this terminology.
\end{rmk}

 In notation and assumption as above, let $(X,r)$ be a finite quadratic set with
 $S=S(X,r)$ graded by length.
 Then the order of the graded component $S_m$ equals the number of $\Dcal_m(r)$-orbits in $X^m$.

\section{The Yang-Baxter algebra $\cA(\textbf{k}, X, r)$  of a finite
nondegenerate symmetric set $(X,r)$}
\label{sec:YBalgebras}

It was proven through the years that the Yang-Baxter algebras $\cA(\textbf{k},
X, r)$ coresponding to finite nondegenerate symmetric sets have remarkable algebraic and
homological properties. They are noncommutative, but have many of the "good"
properties of the commutative polynomial ring
$\textbf{k}[x_1, \cdots, x_n]$, see Remarks \ref{factAS},  \ref{fact1} and Theorem \ref{thm:PBWmain}. This
motivates us to look for more analogues
coming from commutative algebra and algebraic geometry.

\subsection{Basic facts about the YB algebras $\cA(\textbf{k}, X, r)$  of finite
solutions $(X,r)$}
\label{subsec:YBalgebras}

The following remark observes the importance of  finite square-free solutions and their close relations to Artin-Schelter regularity. The results are extracted from
\cite{GIVB, GI04s, GI12} and \cite{rump}.
\begin{rmk}
\label{factAS}
Suppose $(X,r)$ is a square-free nondegenerate and involutive \emph{quadratic set} of order $n$.
Let and  $\cA = \cA(\textbf{k}, X, r)$  be the associated quadratic algebra.
The following conditions are equivalent.
\begin{enumerate}
\item
\label{factAS1}
 $\cA$ is an Artin-Schelter regular PBW algebra.
\item
\label{factAS2}
 $(X,r)$ is a solution of YBE.
 \item
 \label{factAS3}
There exists an enumeration
$X=\{x_1, x_2, \cdots, x_n\}$ such that $\cA$ is a binomial skew polynomial algebra.
\end{enumerate}
\end{rmk}
The implication (\ref{factAS1})$\Longrightarrow$ (\ref{factAS3}) follows from \cite[Theorem 1.2]{GI12}. (\ref{factAS3})  $\Longrightarrow$ (\ref{factAS1}) is proven in \cite[Theorem B]{GI04s} (see also \cite{GIVB}). (\ref{factAS3})  $\Longrightarrow$ (\ref{factAS2}) is proven in \cite[Theorem 1.1]{GIVB}.
The implication (\ref{factAS2}) $\Longrightarrow$ (\ref{factAS3}) was conjectured by the author and proven by Rump, see \cite[Theorem 1]{rump}.

\begin{rmk}
\label{fact:PBW}
Note that among all Yang-Baxter algebras $\cA=\cA(K,  X, r)$ of finite solutions studied in this paper the only PBW algebras  are those corresponding to square-free solutions $(X,r)$. This follows from Theorem \ref{thm:PBWmain} which will be proven in the next subsection.
\end{rmk}

\begin{convention}
\label{rmk:conventionpreliminary1}  Let $(X,r)$ be a finite solution of YBE
of order $n$, and let $\cA = \cA(\textbf{k},X,r)$ be the associated
Yang-Baxter algebra.
(a) If $(X,r)$ is square-free we fix an
enumeration such that $X= \{x_1, \cdots, x_n\}$ is a set of PBW generators
of $\cA$. In this case $\cA$ is a binomial skew polynomial ring, see
 Definition \ref{binomialringdef}.
 (b)  If $(X,r)$ is not square-free we fix an arbitrary enumeration
$X=\{x_1, \cdots, x_n\}$ on $X$.

In each of the cases (a) and (b) we extend the fixed enumeration on $X$ to the
 deg-lex order $<$ on
$\langle X \rangle$.
By convention the Yang-Baxter algebra $\cA =\cA_X= \cA(\textbf{k},X, r)$
is presented as
\begin{equation}
\label{eq:Algebra}
\begin{array}{c}
\cA =\c A(\textbf{k},X,r) = \textbf{k}\langle X  \rangle /(\Re_{\cA})
\simeq \textbf{k}\langle X ; \;\Re(r)
\rangle ,\;\;\text{where}\\
\Re_{\cA}
 = \{xy-y^{\prime}x^{\prime}\mid xy> y^{\prime}x^{\prime},  \; \text{and}\; r(xy)=y^{\prime}x^{\prime}
 \}.
\end{array}
\end{equation}
Consider the two-sided ideal  $I=(\Re_{\cA})$ of $\textbf{k}\langle X  \rangle$,
let $G= G(I)$ be the  unique reduced
Gr\"{o}bner basis of $I$  with respect to $<$. Here we do not need an explicit description of the reduced Gr\"{o}bner basis
$G$  of $I$, but we need some details.
In the case (a) one has $G=\Re_{\cA}$.
It follows from Remark \ref{fact:PBW} that in the case (b)
the set of relations  $\Re_{\cA}$ is not a Gr\"{o}bner basis of $I$, but  $\Re_{\cA} \subsetneqq G$.
Moreover, the shape of the relations  $\Re_{\cA}$ and standard techniques from noncommutative Gr\"{o}bner bases theory imply
that the Gr\"{o}bner basis $G$ is finite, or countably
 infinite, and consists of homogeneous binomials $f_j= u_j-v_j,$ where
 $\LM(f_j)= u_j > v_j$, and $u_j, v_j \in X^m$, for some $m\geq 2$.
The set of all normal monomials modulo $I$ is denoted by $\cN$. As we mentioned
above $\cN= \cN(I) = \cN(G)$.
An element $f \in \textbf{k} \asoX $ is in normal form (modulo $I$), if $f \in
\Span_{\textbf{k}} \cN $.
The free
monoid $\asoX$ splits as a disjoint union
$\asoX=  \cN\sqcup \LM(I).$
The free associative algebra $\textbf{k} \asoX$ splits as a direct sum of
$\textbf{k}$-vector
  subspaces
  $\textbf{k} \asoX \simeq  \Span_{\textbf{k}} \cN \oplus I$,
and there is an isomorphism of vector spaces
$\cA \simeq  \Span_{\textbf{k}} \cN.$
As usual, we denote
\begin{equation}
\label{eq:Nm}
\cN_{d}=\{u\in \cN \mid u\mbox{ has length } d\}.
\end{equation}
 Then
$\cA_d \simeq \Span_{\textbf{k}} \cN_{d}$ for every $d\in\N_0$.
By Corollary \ref{cor:dimAd} $\dim \cA_d  = |\cN_{d}| = \binom{n+d-1}{d}, \; \forall d \geq 0$.
Note that since the set of relations $\Re_{\cA}$ is a finite set of homogeneous polynomials,
the elements
of  the reduced Gr\"{o}bner basis $G = G(I)$  of degree $\leq d$ can be found
effectively, (using the standard strategy for constructing a Gr\"{o}bner
basis) and therefore the set of normal monomials $\cN_{d}$  can be found
inductively for $d=1, 2, 3, \cdots .$
(Here we do not need an explicit description of the reduced Gr\"{o}bner basis
$G$  of $I$).
It follows from Bergman's Diamond lemma, \cite[Theorem 1.2]{Bergman}, that
if we consider the space $\textbf{k} \cN$
endowed with multiplication defined by
 \[f \bullet g := \Nor (fg), \quad \text{for every}\; f,g \in \textbf{k} \cN\]
then
$(\textbf{k} \cN, \bullet )$
has a well-defined structure of a graded algebra, and there is an isomorphism of
graded algebras
\[
\cA=\cA(\textbf{k}, X, r) \cong (\textbf{k} \cN, \bullet ),\;\; \text{so}\; \; \cA =\bigoplus_{d\in\N_0}  A_d \cong \bigoplus_{d\in\N_0}  \textbf{k}\cN_d.
\]
By convention we shall often identify the algebra  $\cA$ with $(\textbf{k}\cN,
\bullet )$.

In the case (a) when $(X,r)$ is square-free,  the set of normal monomials is exactly
$\cT$ (the set of ordered terms in $X$),  so $\cA$ is identified with $(\textbf{k} \cT, \bullet )$ and
$S(X,r)$ is identified with $(\cT, \bullet)$.
\end{convention}

We shall recall more properties of the Yang-Baxter algebras which
will be used in the sequel, but first we need the following  lemma which
is involved in our interpretation of  \cite[Theorem 1.3]{GIVB}) as Remark \ref{fact1}.

\begin{lem}
\label{lem:ore}
Every nondegenerate involutive quadratic set $(X, r)$ satisfies the following condition (\emph{Ore condition}).
\begin{equation}
\label{eq:condition}
\begin{array}{l}
\text{Given $a,b\in X$ there exist unique $c,d\in X$ such that $r(ca)= db$.}\\
\text{ Furthermore if $a=b$ then $c=d$}.
\end{array}
\end{equation}
In particular, $r$ is 2-cancellative.
\end{lem}

\begin{proof}
Let $(X,r)$ be a nondegenerate involutive quadratic set (not necessarily
finite).
Let $a, b \in X$. We have to find unique pair $c, d$, such that $r(c,a) = (d,
b)$. By the nondegeneracy there is unique $c\in X$, such that $c^a=b$.  Let
$d = {}^ca$,  then $r(c,a) = ({}^ca, c^a)= (d,b),$ as desired. It also follows
from the nondegeneracy that the pair $c, d$ with this propery is unique.
Assume now that $a=b$. The equality $r(c,a) =(d,a)$ implies $ ({}^ca, c^a)=
(d,a)$, so $c^a=a$. But $r$ is involutive, thus $(c,a) = r(d,a) =({}^da, d^a)$,
and therefore $d^a=a$. It follows that $c^a=d^a$, and, by the nondegeneracy, $c
= d$.
\end{proof}

The following facts are a compilation of results from \cite{GIVB} and are true for every finite solution of YBE.
\begin{rmk}
\label{fact1}
Suppose $(X,r)$ is a finite solution of YBE of order $n$, $X = \{x_1,
\cdots , x_n\}$. Let $S=S(X,r)$ be the associated Yang-Baxter monoid and let $\cA = \cA(\textbf{k},
X, r)$ be the associated Yang-Baxter algebra.
Then the following conditions hold.
\begin{enumerate}
\item  (A modified version of \cite[Theorem 1.3]{GIVB})

$S$ is a semigroup of $I$-type, that is there is a bijective map $v: \cU
\mapsto S$, where $\cU$ is the free $n$-generated abelian monoid
$\cU=[u_1, \cdots, u_n]$ such that $v(1) = 1,$ and such that
\[
\{v(u_1a), \cdots, v(u_na)\} = \{ x_1v(a), \cdots, x_nv(a) \}, \;\text{for
all}\; a \in \cU.
\]
\item The Hilbert series of $\cA$ is $H_{\cA}(t)= 1/(1-t)^n.$
\item \cite[Theorem 1.4]{GIVB}
(a) $\cA$ has finite global dimension and polynomial growth; (b) $\cA$ is Koszul;
 (c) $\cA$ is left and right Noetherian; (d) $\cA$ satisfies the Auslander condition and
 is Cohen-Macaulay.
(e) $A$ is finite over its center.
\item \cite[Corollary 1.5]{GIVB}
$\cA$ is a domain, and in particular the monoid $S$ is cancellative.
\end{enumerate}
\end{rmk}
For convenience of the reader we shall make a brief observation.
Note first that the hypothesis of Remark \ref{fact1} is satisfied by arbitrary finite solution of YBE $(X,r)$ which is not necessarily square free, and, in general, the algebra
$\cA = \cA(\textbf{k},
X, r)$ is not a binomial skew polynomial ring, or equivalently,  $\cA$ is not a  PBW algebra.

 Next observe that part (1) of Remark \ref{fact1} is a modification of the original second part of \cite[Theorem 1.3]{GIVB} which  states (in our terminology):
"Suppose that  $(X,r)$ is a finite solution of YBE of order $n$ which satisfies the condition (\ref{eq:condition}).
Then the monoid $S(X,r)$ is of $I$ type".
However, under the hypothesis of Remark \ref{fact1}, Lemma \ref{lem:ore} implies the necessary condition (\ref{eq:condition}).

The following corollary is straightforward from Remark \ref{fact1} (1) and will be used throughout the paper.
\begin{cor}
\label{cor:dimAd}
In notation and conventions as above. Let $(X,r)$ be a finite solution of YBE. Then for every integer $d \geq 1$ there are equalities
\begin{equation}
\label{eq:dimAd}
\dim \cA_d=\binom{n+d-1}{d} = \binom{n+d-1}{n-1}=|\cN_d|.
\end{equation}
\end{cor}

\subsection{Every finite solution $(X,r)$ whose Yang-Baxter algebra $\cA(\textbf{k}, X, r)$ is PBW is square-free}
\label{subsec:YBalgebras}
In this subsection we give an answer to Problems \ref{problem} (1).

Suppose $(X,r)$ is a finite solution of YBE whose Yang-Baxter
algebra
$A=A(\textbf{k},  X, r)$ is PBW, where
$X = \{x_1, x_2, \cdots, x_n\}$  is a set of PBW generators. Then
$\cA=\textbf{k}\asX/(\Re_{\cA})$, where
the set of (quadratic) defining relations $\Re_{\cA}$ of $\cA$ coincides with the reduced
 Gr\"{o}bner basis of the ideal $(\Re_{\cA})$ modulo the deg-lex order on $\asX$. The cardinality of $\Re_{\cA}$ is exactly $\binom{n}{2}$, see Corollary \ref{cor:orbits_and_fixedpoints}.
Recall that the set of leading monomials
\begin{equation}
\label{eq:lempbw0}
W =\{LM(f) \mid  f \in \Re_{\cA}\}
\end{equation}
is called \emph{the set of obstructions} for $\cA$, in the sense of Anick, \cite{Anickmon}.
\begin{lem}
\label{lem:AisPBW}
Suppose $(X,r)$ is a solution of YBE of order $n$ and that its
Yang-Baxter algebra $\cA=\cA(\textbf{k},  X, r)$ is PBW, where
$X=  \{x_1, x_2, \cdots , x_n\}$ is a set of PBW generators. Then there exists a
permutation
\[y_1=x_{s_1},\; y_2=x_{s_2},\; \cdots,\;  y_n=x_{s_n} \;\text{of} \; \; x_1, x_2, \cdots , x_n,   \]
such that the following conditions hold.
\begin{enumerate}
\item The set of obstructions
$W =\{LM(f) \mid f \in \Re_{\cA}\}$
consists of $\binom{n}{2}$ monomials given below
\begin{equation}
\label{eq:lempbw1}
W =\{y_jy_i \mid 1 \leq i < j \leq n\}.
\end{equation}
\item
 The normal $\textbf{k}$-basis of $\cA$ modulo $I =(\Re_{\cA})$ is the set
\begin{equation}
\label{eq:lempbw2}
\cN = \{y_1^{\alpha_1}y_2^{\alpha_2}\cdots y_n^{\alpha_n} \mid \alpha_i \geq 0,
\; \text{for}\; 1 \leq i \leq n\}.
\end{equation}
\end{enumerate}
\end{lem}
\begin{proof}
Let $W$ be the set of obstructions defined via (\ref{eq:lempbw0}) and let
 $A_W$ be \emph{the associated monomial algebra} defined as
\begin{equation}
\label{eq:lempbw3}
A_W := \textbf{k}\asX/(W)
\end{equation}
It is well known that a word $u\in \asX$ is normal modulo $I=(\Re_{\cA})$ \emph{iff} $u$ is normal modulo the set of obstructions $W$.
Therefore the two algebras $\cA$ and $A_W$ share the same normal $\textbf{k}$-basis   $\cN =\cN(I) =\cN(W)$ and their Hilbert series are equal.
By  Remark \ref{fact1} part (2), the Hilbert series of $\cA$ is $H_{\cA}(t) = 1/(1-t)^n,$  therefore
\begin{equation}
\label{eq:lempbw4}
 H_{A_W}(t)= H_{\cA}(t) = 1/(1-t)^n.
\end{equation}
Thus the Hilbert series of  $A_W$ satisfies condition (5) of \cite[Theorem 3.7]{GI12} (see page 2163), and  it follows from the theorem
that there exists a
permutation $y_1=x_{s_1}, \; y_2= x_{s_2}, \; \cdots,\;  y_n= x_{s_n}$ of the generators $x_1, x_2, \cdots , x_n$, such that
the set of obstructions $W$ satisfies (\ref{eq:lempbw1}).
The Diamond Lemma,  \cite{Bergman}, and the explicit description  of the obstruction set $W$ given in (\ref{eq:lempbw1}) imply that
the set of normal words $\cN=\cN(I) = \cN(W)$ is described in   (\ref{eq:lempbw2}).
\end{proof}
Observe that if the permutation given in the lemma is not trivial then there is an inversion, that is a  pair $i, j$ with $i<j$ and $y_j< y_i$.
\begin{thm}
    \label{thm:PBWmain}
Suppose $(X,r)$ is a solution of YBE of order $n$, and $\cA=\cA(\textbf{k},  X, r)$ is its Yang-Baxter algebra.
Then $\cA$ is a PBW algebra with a set of PBW generators $X = \{x_1, x_2, \cdots, x_n\}$  (enumerated properly)
  if and only if $(X,r)$ is a square-free solution.
\end{thm}
\begin{proof}

If $(X,r)$ is square-free then there exists an enumeration
$X=\{x_1, \cdots, x_n\}$, so that $\cA$ is a binomial skew-polynomial ring in the sense of \cite{GI96},  and therefore $\cA$  is PBW.
This was a conjecture of the author which was proven later by Rump, \cite[Theorem 1]{rump}.

Assume now that $(X,r)$ is a finite solution of order $n$ whose Yang-Baxter algebra $\cA=\cA(\textbf{k},  X, r)$ is PBW, where
$X = \{x_1, x_2, \cdots, x_n\}$  is a set of PBW generators.  We have to show that $(X,r)$ is square-free that is $r(x,x) = (x, x)$ , for all $x \in X.$

It follows from our assumptions that in the presentation $\cA=\textbf{k}\asX/(\Re_{\cA})$
the set of (quadratic) defining relations $\Re_{\cA}$ of $\cA$ is the reduced
 Gr\"{o}bner basis of the ideal $(\Re_{\cA})$ modulo the deg-lex order on $\asX$.
By Lemma \ref{lem:AisPBW} there exists a
permutation $y_1=x_{s_1}, y_2= x_{s_2}, \cdots , y_n= x_{s_n}$ of    $x_1, x_2, \cdots , x_n$ such that
the obstruction set $W=\{\LM (f)\mid f \in \Re_{\cA}\}$  satisfies (\ref{eq:lempbw1}) and the set of normal monomials
$\cN$ described in (\ref{eq:lempbw2})  is a PBW basis of $\cA$.

We use some properties of $(X,r)$ and the relations of $\cA$ listed below.

(i) $(X,r)$ is $2$-cancellative. This follows from Lemma \ref{lem:ore};
(ii) There are exactly $n$ fixed points $xy\in X^2$ with $r(x,y) = (x,y)$ and the set $\Re_{\cA}$ consists of exactly  $ \binom{n}{2}$ relations.  This follows from Corollary \ref{cor:orbits_and_fixedpoints}.
(iii)  Every monomial of the shape $y_jy_i,\;  1\leq i < j\leq n$ is the leading monomial of some polynomial $\varphi_{ji}\in \Re_{\cA}.$
(It is possible that $y_j < y_i$ for some $j> i$.)

 Therefore the algebra
$\cA$  has a presentation
 \[\cA=\textbf{k} \langle x_1, \cdots , x_n\rangle/(\Re_{\cA})\] with
precisely $\binom{n}{2}$ defining relations
\begin{equation}
\label{eq:pbw11}
\Re_{\cA}=\{\varphi_{ji} = y_jy_i - u_{ij}
 \mid 1\leq i<j\leq n\}
\end{equation}
 such that
\begin{enumerate}
\item For every pair $i, j, \; 1\leq
i<j\leq n$, the monomial $u_{ij}$ satisfies $u_{ij}= y_{i^{\prime}} y_{j^{\prime}},$
where $i^{\prime} \leq  j^{\prime}$, and $y_j > y_{i^{\prime}}$ (since $\LM(\varphi_{ji}) =y_jy_i > y_{i^{\prime}} y_{j^{\prime}} $, and since $(X,r)$ is $2$-cancellative);
\item Each  monomial $y_iy_j$  with $1 \leq i\leq j \leq n$  occurs at most once in $\Re_{\cA}$ (since $r$ is a bijective map).
\item $\Re_{\cA}$ is the
{\it reduced Gr\"obner basis} of the two-sided ideal $(\Re_{\cA})$,
with respect to the degree-lexicographic order $<$ on $\asX$.
\end{enumerate}
In terms of the relations $\Re_{\cA}$ our claim that $r(x,x) = (x,x)$,  for all $x\in X$,   is equivalent to
\begin{equation}
\label{eq:pbw21}
 u_{ij} \neq xx, \; \; \text{where}\;\; x \in X, \;\; \text{and} \; 1 \leq i < j \leq n.
\end{equation}
So far we know that $(X,r)$ has exactly $n$ fixed points, and each monomial $y_jy_i,  1\leq i < j\leq n$
  is not a fixed point.
Therefore it will be enough to show that a monomial $y_iy_j,$ with $1 \leq i <  j \leq  n$, can not be a fixed point.

Assume on the contrary, that $r(y_i, y_j) = (y_i, y_j)$, for some $1 \leq i < j \leq  n$. We claim that in this case $\Re_{\cA}$ contains two relations of the shape
\begin{equation}
\label{eq:pbw31}
 (a) \quad y_py_q -y_jy_j, \; \text{where}\;  p>q,  y_p > y_j,\; \; \text{and}\; \;  (b) \quad y_sy_t -y_iy_i,
\;  \text{where}\;  s>t ,  y_s > y_i.
\end{equation}
 Consider the increasing chain of left ideals of $\cA$
\[
I_1 \subseteq I_2\subseteq \cdots \subseteq I_k \subseteq \cdots,
\]
where for $k \geq 1,$  $I_k$ is the left ideal
\[
I_k= {}_{\cA}  (y_iy_j, \; y_iy_j^2, \; \cdots,\;  y_iy_j^k ).
\]
By \cite[Theorem 1.4]{GIVB}, see also Remark \ref{fact1} (3),  the algebra $\cA$ is left Noetherian
hence there exists $k > 1$, such that $I_{k-1}= I_{k}= I_{k+1}= \cdots$, and therefore $y_iy_j^k  \in I_{k-1}$.
This implies
\begin{equation}
\label{eq:pbw41}
w\bullet (y_iy_j^c) = y_iy_j^k  \in \cN, \;\text{for some}\; c,  \; 1 \leq c \leq k-1, \;\text{and some}\; w \in \cN,  |w|=k-c.
\end{equation}
It follows from (\ref{eq:pbw41}) that the monomial $v_0 = y_iy_j^k$
can be obtained from the monomial $w(y_iy_j^c)$ by applying a finite sequence of replacements (reductions) in $\asX$. More precisely,  there exists a sequence of monomials
\[
v_0 =  y_iy_j^k, \; v_1, \cdots, \; v_{t-1}, \; v_t =w (y_iy_j^c) \in \asX
\]
and replacements
\begin{equation}
\label{eq:pbw51}
 v_t\rightarrow v_{t-1}\rightarrow \cdots \rightarrow v_1 \rightarrow   v_0 =  y_iy_j^k\in \cN,
\end{equation}
where each replacement comes from some quadratic relation $f_{pq}= y_py_q-u_{qp}$ in (\ref{eq:pbw11}) and has the shape
 \[a[y_py_q]b \rightarrow  a(u_{qp})b, \; \; \text{where}\; \; n\geq p>q\geq 1, \; a, b \in  \asX .\]
We have assumed that $ y_iy_j$ is a fixed point, so it can not occur in a relation in (\ref{eq:pbw11}). Thus the  rightmost replacement in (\ref{eq:pbw51})
is of the form
\[u_1 = y_iy_j \cdots y_j [y_py_q]\cdots  y_j  \rightarrow y_iy_j \cdots y_j (u_{qp})\cdots  y_j=
y_iy_j\cdots y_j (y_jy_j)\cdots  y_j = v_0
 \]
where $p, q$ is a pair with, $1 \leq q < p \leq n,$ $u_{qp}=  y_jy_j $ and $y_p > y_j$. In other words the set
$\Re_{\cA}$ contains a relation of type (a)  $y_py_q - x_jx_j$ where $p>q,  y_q> y_j.$

Analogous argument proves the existence of a relation of the type (b) in (\ref{eq:pbw31}). This time we consider an increasing chain of right ideals $I_1 \subseteq I_2\subseteq \cdots \subseteq I_k \subseteq \cdots$, where $I_k$ is the right ideal
$I_k= (y_iy_j, y_i^2y_j, \cdots , y_i^k y_j )_{\cA}$ and apply the right Noetherian property of $\cA$.

Consider now the subset of fixed points
\[
\Fcal_0 (X,r) = \{y_iy_j  \in X^2  \;\;  \text{such that } \;\;  i< j \;  \text{and } \; r(y_i, y_j) = (y_i, y_j)\},
\]
which by our assumtion is not empty.
Then $\Fcal_0 (X,r) $ has cardinality $m\geq 1$ and  $\Re_{\cA}$ contains at least $m+1$ (distinct) relations
of the type
\begin{equation}
\label{eq:fixed12}
y_py_q  - xx, \; \text {where}\;x \in X,  p> q \; \text {and}\; y_p > x.
\end{equation}
The set $\cN_2$ of normal monomials of length $2$   contains $\binom{n}{2}$  elements of the shape
$y_sy_t,  1 \leq s < t \leq n$, and we have assummed that $m$ of them are fixed.
Then there are $\binom{n}{2}-m$ distinct monomials $y_iy_j \in \cN_2, 1 \leq i < j\leq n $  which are not fixed.
 Each of these monomials occurs in exactly one relation
\[y_sy_t -  y_iy_j
,\;\text{where}\;
 r(y_s, y_t) = (y_i, y_j), \;   s>t, \; y_s > y_i.\]
Thus $\Re_{\cA}$  contains $ \binom{n}{2}-m $ distinct square-free relations and at least $m+1$ relations which contain squares as in (\ref{eq:fixed12}) .
Therefore the set of relations has cardinality
\[|\Re_{\cA}|\geq \binom{n}{2}-m+ m+1 > \binom{n}{2}  \]
which is a contradiction.

We have shown that a monomial $y_iy_j$ with $1 \leq i < j \leq n$ can not be a fixed point, and therefore it occurs in a relation in $\Re_{\cA}.$
But $(X,r)$ has exactly $n$ fixed points, so these  are the elements of the diagonal of $X^2$ ,  $x_ix_i, 1 \leq i \leq n$. It follows that $(X,r)$ is square-free.
\end{proof}

\begin{pro}
    \label{thm:PBW}
Let $(X,r)$ be a finite non-degenerate involutive quadratic set, and let $\cA=\cA(\textbf{k},  X, r) = \textbf{k}\asX/ (\Re_{\cA})$  be its quadratic algebra. Assume that there is an enumeration $X = \{x_1, x_2, \cdots, x_n\}$  of $X$ such that the set
\[
\cN = \{x_1^{\alpha_1}x_2^{\alpha_2}\cdots x_n^{\alpha_n}\mid \alpha_i \geq 0
\; \text{for}\; 1 \leq i \leq n\}
\]
is a normal $\textbf{k}$-basis of $\cA$ modulo the ideal $I= (\Re_{\cA})$.
Then
$\cA=\cA(\textbf{k},  X, r)$ is a PBW algebra, where
$X = \{x_1, x_2, \cdots, x_n\}$  is a set of PBW generators of $\cA$
and the set of relations $\cR_0$ is a quadratic Gr\"obner basis of the two-sided ideal $(\Re_{\cA})$.
The following conditions are equivalent.
\begin{enumerate}
\item[(1)] The algebra $\cA$ is left and right Noetherian.
\item[(2)]  The quadratic set $(X,r)$ is square-free.
\item[(3)]  $(X,r)$ is a solution of YBE.
\item[(4)]  $\cA$ is a binomial skew polymomial ring in the sense of \cite{GI96}.
\end{enumerate}
\end{pro}
\begin{proof}
The quadratic set $(X,r)$ and the relations of $\cA$  satisfy conditions similar to those listed in the proof of Theorem \ref{thm:PBWmain}. More precisely:
(i) $(X,r)$ is $2$-cancellative. This follows from Lemma \ref{lem:ore};
 (ii) There are exactly $n$ fixed points $xy\in X^2$ with $r(x,y) = (x,y)$. This follows from Corollary \ref{cor:orbits_and_fixedpoints}.
(iii) It follows from the hypothesis that every monomial of the shape $x_jx_i, 1\leq i < j\leq n,$ is not in the normal  $\textbf{k}$ -basis  $\cN$, and therefore it is the highest monomial of some polynomial $\varphi_{ji}\in \Re_{\cA}.$
 \cite[Proposition 2.3.]{GI11} implies that if $(X,r)$ is a nondegenerate involutive quadratic set  of order $n$ then the set $\Re_{\cA}$ consists of exactly  $ \binom{n}{2}$ relations.
Therefore the algebra
$\cA$  has a presentation
 \[\cA=\textbf{k} \langle x_1, \cdots , x_n\rangle/(\Re_{\cA})\] with
precisely $\binom{n}{2}$ defining relations
\begin{equation}
\label{eq:pbw1}
\Re_{\cA}=\{\varphi_{ji} = x_{j}x_{i} - x_{i'}x_{j'}
 \mid 1\leq i<j\leq n\}
\end{equation}
 such that
\begin{enumerate}
\item[(a)] For every pair $i, j, \; 1\leq
i<j\leq n$, one has
 $i^{\prime} \leq  j^{\prime}$, and $j > i^{\prime}$ (since $\LM(\varphi_{ji}) =x_jx_i> x_{i'}x_{j'}$, and since $(X,r)$ is $2$-cancellative);
\item[(b)]  Each ordered monomial (term) of length 2 occurs at most once in $\Re_{\cA}$ (since $r$ is a bijective map).
\item[(c)]  $\Re_{\cA}$ is the
{\it reduced Gr\"obner basis} of the two-sided ideal $(\Re_{\cA})$,
with respect to the deg-lex order $<$ on $\asX$,  or
equivalently the overlaps $x_kx_jx_i,$ with $k>j>i$ do not give
rise to new relations in $\cA.$
\end{enumerate}
(1) $\Rightarrow$ (2). The proof is analogous to the proof of Theorem  \ref{thm:PBWmain}. It is enough to show that a monomial $x_ix_j$ with $1 \leq i <  j \leq  n,$ can not be a fixed point. Assuming the contrary, and
applying an argument similar to the proof of Theorem  \ref{thm:PBWmain}, in which we involve the left and right Noetherian properties of $\cA$, we get a contradiction. Thus every monomial
 $x_ix_j$ with $1 \leq i < j \leq n$  occurs in a relation in $\Re_{\cA}.$ At the same time the monomials $x_jx_i$ with $1 \leq i < j \leq n$ are also involved in the relations  $\Re_{\cA}$, hence they are not fixed points.
But $(X,r)$ has exactly $n$ fixed points, so these  are the elements of the diagonal of $X^2$ ,  $x_ix_i, 1 \leq i \leq n$. It follows that $(X,r)$ is square-free.
(2) $\Rightarrow$ (4). If $(X,r)$ is square-free then the relations $\Re_{\cA}$ given in (\ref{eq:pbw1})  are exactly the defining relations of a binomial skew polynomial ring, moreover by the hypothesis of the proposition condition (d$^{\prime}$) in see Definition \ref{binomialringdef} therefore all conditions in Definition \ref{binomialringdef} hold, so $\cA$ is a skew polynomial ring with binomial relations in the sence of
\cite{GI96}.
The implication (4) $\Rightarrow$  (3)  follows from \cite[Theorem 1.1]{GIVB}.
The implication (3) $\Rightarrow$  (1) follows from \cite[Theorem 1.4]{GIVB}, see also Remark \ref{fact1} (3).
\end{proof}

\section{The $d$-Veronese subalgebra $\cA^{(d)}$ of the Yang-Baxter algebra $\cA (\textbf{k}, X, r)$,
its generators and relations}
\label{sec:dVeronese}

In this section $(X,r)$ is a finite solution of YBE, $d\geq 2$ is an integer.
We shall study the $d$-Veronese subalgebras $\cA^{(d)}$ of  the  Yang-Baxter
 algebra $\cA=\cA(\textbf{k}, X, r)$.
 This is an
algebraic construction which mirrors the Veronese embedding.
Results on Veronese subalgebras of noncommutative graded algebras
appeared first in \cite{Froberg} and \cite{Backelin}.
Our main reference here is  \cite[Section 3.2]{PoPo}.
We shall prove Theorem \ref{thm:d-Veronese_relations} which
 presents the $d$-Veronese subalgebra $\cA^{(d)}$ in terms of generators and
 quadratic relations.

\subsection{Veronese subalgebras of graded algebras}
We recall first some basic definitions and facts about Veronese subalgebras of
general graded algebras.
\begin{dfn}
Let $A= \bigoplus_{m\in\N_0}A_{m}$ be a graded $\textbf{k}$-algebra. For any integer $d\geq
1$, the \emph{$d$-Veronese subalgebra} of $A$ is the graded algebra
    \[A^{(d)}=\bigoplus_{m\in\N_0} A_{md}.\]
\end{dfn}

By definition the algebra $A^{(d)}$ is a subalgebra of $A$. However, the
embedding is not a graded algebra morphism.
The Hilbert function of $A^{(d)}$ satisfies
    \[h_{A^{(d)}}(t)=\dim(A^{(d)})_t=\dim(A_{td})=h_A(td).\]

It follows from \cite[Proposition 2.2, Ch 3]{PoPo}
that if $A$ is a one-generated quadratic Koszul algebra, then its Veronese
subalgebras are also one-generated quadratic and Koszul.

\begin{cor}
\label{cor:dVeron}
Let $(X,r)$ be a solution of order $n$, and let $\cA = \cA(\textbf{k}, X,r)$ be
its Yang-Baxter algebra, let $d \geq 2$ be an integer.
Then the  $d$-Veronese subalgebra $\cA^{(d)}$  is one-generated, quadratic and
Koszul.
\end{cor}
\begin{proof}  If $(X,r)$ is a solution
of order $n$ then, by definition the Yang-Baxter algebra $\cA = \cA(\textbf{k},
X,r)$  is one-generated and quadratic. Moreover, $\cA$ is Koszul, see
Remark \ref{fact1}.  It follows straightforwardly from \cite[Proposition
2.2, Ch 3]{PoPo} that $\cA^{(d)}$  is one-generated, quadratic and Koszul.
\end{proof}
We shall prove in the next section that $\cA^{(d)}$  is a left and a right Noetherian domain.

In the assumptions of Corollary \ref{cor:dVeron}, it is clear, that the d-Veronese
subalgebra $\cA^{(d)}$ satisfies
\begin{equation}
\label{eq:A^d}
\cA^{(d)} = \bigoplus_{m\in\N_0}  A_{md} \cong \bigoplus_{m\in\N_0}
\textbf{k}\cN_{md}.
\end{equation}
Moreover, the normal monomials
$w \in \cN_d$ of length $d$ are degree one generators of $A^{(d)}$, and by
Corollary \ref{cor:dimAd}
there are equalities
\[
|\cN_d | =\dim \cA_d=\binom{n+d-1}{d}
\]
We set
\begin{equation}
\label{eq:N}
N=\binom{n+d-1}{d}
\end{equation}
and order the elements of $\cN_d$
lexicographically:
        \begin{equation}
        \label{eq:deg-d-monomials}
\cN_d := \{ w_1 < w_2< \dots < w_N \}.
        \end{equation}
The $d$-Veronese $\cA^{(d)}$ is a quadratic algebra with
one-generators $w_1, w_2, \dots, w_N.$ We shall find a
minimal set of quadratic relations for $\cA^{(d)},$  each of which is a linear
combination of products  $w_iw_j$ for some $i,j\in\{1,\dots,N\}$.
The relations are intimately connected with the properties of the Yang-Baxter monoid
$S(X,r).$
As a first step we shall introduce a finite nondegenerate symmetric set $(S_{d},
r_d)$ induced in a natural way by the braided monoid $S(X,r)$.

\subsection{The braided monoid $S = S(X, r)= (S, r_S)$  of a braided set $(X,r)$}
Matched pairs of monoids, M3-monoids and braided monoids in a most general
setting were studied in \cite{GIM08}, where the interested reader can find the
necessary definitions and
the original results. Here we extract only some facts which will be used in the
paper.
\begin{fact}(\cite{GIM08}, Theor. 3.6, Theor. 3.14.)
\label{theoremA}
Let  $(X,r)$ be a braided set and let
$S=S(X,r)$ be its Yang-Baxter monoid.
Then
\begin{enumerate}
\item
The left and the
right actions
$
{}^{(\;\;)}{\circ}: X\times X  \longrightarrow
 X  , \; \text{and}\;\; \circ^{(\;\;)}: X
\times X \longrightarrow  X
$
defined via $r$ can be extended in a unique way to a left and a
right action
\[{}^{(\;\;)}{\circ}: S\times S  \longrightarrow
 S,\; \; (a, b) \mapsto  {}^ab, \; \text{and}\;\; \circ^{(\;\;)}: S
\times S \longrightarrow  S, \;\; (a, b) \mapsto  a^b
\]
which make $S$ \emph{a strong graded {\bf M3}-monoid}.
In particular, the following equalities hold in $S$ for all $a, b, u, v \in S$.
\begin{equation}
\label{eq:braided_monoid}
\begin{array}{lclc}
{ML0 :}\quad & {}^a1=1,\quad  {}^1u=u;\quad &{MR0:} \quad &1^u=1,\quad a^1=a
\\
 {ML1:}\quad& {}^{(ab)}u={}^a{({}^bu)},\quad& {MR1:}\quad  & a^{(uv)}=(a^u)^v
 \\
{ML2:}\quad & {}^a{(u.v)}=({}^au)({}^{a^u}v),\quad &{MR2:}\quad &
(a.b)^u=(a^{{}^bu})(b^u)\\
{M3:}\quad &{}^uvu^v=uv.& &
\end{array}
\end{equation}
These actions define a bijective map \[r_S: S\times S
\longrightarrow S\times S, \quad  r_S(u, v) := ({}^uv, u^v)\]
which obeys the Yang-Baxter equation, so $(S, r_S)$ is \emph{a braided
monoid}.
In particular,
$(S, r_S)$
is a set-theoretic solution of YBE, and the associated
bijective map $r_S$ restricts to $r$.
\item The following conditions hold.
\begin{enumerate}
\item $(S,r_S)$ is \emph{a graded braided
monoid},  that is the actions agree with the grading (by length) of $S$:
\begin{equation}
\label{eq:braided_monoid2}
|{}^au|= |u|= |u^a|,  \forall   \; a,u \in S.
\end{equation}
\item $(S, r_S)$ is non-degenerate  \emph{iff} $(X,r)$ is non-degenerate.
    \item $(S, r_S)$ is involutive \emph{iff}
$(X,r)$ is involutive. \item $(S, r_S)$ is square-free \emph{iff}
$(X,r)$ is a trivial solution.
\end{enumerate}
\end{enumerate}
\end{fact}

Let  $(X,r)$ be a non-degenerate symmetric set, let
$(S, r_S)$  be the associated graded braided monoid, where we consider the
natural grading by length given in (\ref{eq:Sgraded}):
\[S = \bigsqcup_{d\in\N_0}  S_{d},\
S_0 = \{1\},  S_1 = X,\ \mbox{and} \    S_{k}S_{m} \subseteq S_{k+m}.\]
Each of the graded components $S_d, \; d \geq 1$, is $r_S$-invariant. Consider
the restriction
$r_d = (r_S)_{|S_d\times S_d}$, where $r_d$ is the map
$r_d: S_d\times S_d\longrightarrow S_d\times S_d$.

\begin{cor}
\label{dVeron_monoid}
Let $(X,r)$ be a solution of YBE. Then the following conditions hold.
\begin{enumerate}
\item
For every positive integer $d \geq 1$, $(S_d, r_d)$ is a solution of YBE (a nondegenerate
symmetric set).
Moreover, if $(X,r)$ is of finite order $n$, then $(S_d, r_d)$ is a finite solution of YBE of order
\begin{equation}
\label{eq:fixedwords2}
   |S_d|  = \binom{n+d-1}{d} = N.
\end{equation}
\item  The number of fixed points is
$|\Fcal (S_d, r_d)|=  N.$
\end{enumerate}
\end{cor}
\begin{dfn}
\label{def:VeroneseSol}
We call  $(S_d,r_d)$ \emph{the monomial $d$-Veronese solution associated with
$(X,r)$}.
\end{dfn}
The monomial $d$-Veronese solution $(S_d, r_d)$ depends only on the map $r$
and on the integer $d$, it is invariant with respect to the enumeration of $X$.
Although $(S_d, r_d)$ is intimately connected with the $d$-Veronese subalgebra
$\cA^{(d)}$ and its quadratic relations, this solution is not convenient for an
explicit description of the relations. Its rich structure inherited from the
braiding in $(S,r_S)$ is used in the proof of Theorem \ref{thm:Sd_square-free}.
The solution $(S_d, r_d)$  induces in a natural way an isomorphic solution
$(\cN_{d},
\rho_d)$ and the fact that $\cN_{d}$ is ordered lexicographically makes this
solution convenient for our description of the  relations of $\cA^{(d)}$. Note
that the set $\cN_{d}$, as a subset of the set of normal monomials $\cN$, depends
on the initial enumeration of $X$.  We shall construct $(\cN_{d},
\rho_d)$ below.

\begin{rmk}
    \label{rmk:the actions}
Note that given the monomials $a = a_1a_2 \cdots a_p \in X^p$, and $b=
b_1b_2\cdots b_q \in X^q$ we can find effectively the monomials  ${}^{a}{b}\in
X^q$ and $a^b \in X^p$. Indeed, as in \cite{GIM08}, we use the conditions
(\ref{eq:braided_monoid}) to extend the left and the right actions inductively:

\begin{equation}
\label{eq:effective_actions}
\begin{array}{lll}
{}^c{(b_1b_2\cdots b_q)}=({}^cb_1)({}^{c^{b_1}}b_2) \cdots ({}^{(c^{({b_1}\cdots
b_{q-1})})}b_q)),\quad \text{for all $c \in X$}\\
&&\\
{}^{(a_1a_2 \cdots a_p)}b= {}^{a_1}{({}^{(a_2 \cdots a_p)}b)}.
\end{array}
\end{equation}
We proceed similarly with the right action.
\end{rmk}

\begin{lem}
   \label{lem:Daction}
Notation as in Remark  \ref{orbitsinG}.
   Suppose $a, a_1\in X^p,  a_1 \in \Ocal_{\Dcal_p} (a)$, and $b, b_1 \in X^q,
   b_1 \in \Ocal_{\Dcal_q} (b)$,
\begin{enumerate}
\item The following are equalities of words in the free monoid $\asX$:
\begin{equation}
  \label{eq:Noractions}
  \Nor ({}^{a_1}{b_1}) = \Nor ({}^{a}{b}), \quad \Nor ({a_1}^{b_1}) = \Nor
  ({a}^{b}).
\end{equation}
In partricular, if $a, a_1\in X^p$ and $b, b_1 \in X^q$ the equalities $a=a_1$
in $S$ and $b=b_1$ in $S$ imply that
${}^{a_1}{b_1} ={}^ab$ and $a_1^{b_1}= a^b$ hold in $S$.
\item  The following are equalities in the monoid $S$:
 \begin{equation}
  \label{eq:M3Noractions}
  ab = {}^{a}{b}a^b =  \Nor ({}^{a}{b})\Nor ({a}^{b}).
  \end{equation}
\end{enumerate}
\end{lem}
\begin{proof}
By  Remark  \ref{orbitsinG} there is an equality $a = a_1$ in $S$  \emph{iff}
$a_1 \in \Ocal_{\Dcal_p} (a)$, in this case $\Ocal_{\Dcal_p} (a)=\Ocal_{\Dcal_p}
(a_1)$. At the same time  $a =a_1$ in $S$ \emph{iff} $\Nor(a_1) = \Nor(a)$ as
words in $X^p$, in particular,
$\Nor(a) \in  \Ocal_{\Dcal_p} (a)$.
Similarly, $b_1 = b$ in $S$  \emph{iff}  $b_1 \in \Ocal_{\Dcal_q} (b)$, and in
this case $\Nor(b) = \Nor(b_1)\in \Ocal_{\Dcal_q} (b)$.
Part (1) follows from the properties of the actions in $(S, r_S)$ studied in
\cite{GIM08}, Proposition 3.11.

(2) $(S, r_S)$ is an M3- braided monoid, see Fact  \ref{theoremA}, so condition
M3 implies  the first equality in (\ref{eq:M3Noractions}).
Now  (\ref{eq:Noractions}) implies the second equality in
(\ref{eq:M3Noractions}).
\end{proof}

\begin{defnotation}
\label{def:rho}
In notation and conventions as above. Let $d\geq 1$ be
an integer. Suppose $(X,r)$ is a solution of order $n$, $\cA =
\cA(\textbf{k}, X,r)$, is the associated Yang-Baxter algebra, and $(S, r_S)$ is
the associated braided monoid.
By convention we identify $\cA$ with $(\textbf{k} \cN, \bullet )$ and $S$ with
$(\cN, \bullet )$.
Define a left "action" and a right "action" on $\cN_{d}$
as follows.
\begin{equation}
  \label{eq:actions}
  \begin{array}{lll}
\la : \cN_{d} \times \cN_{d} \longrightarrow \cN_{d},   & a \la  b := Nor
({}^ab)\in \cN_{d}, & \forall a, b \in \cN_{d}\\
  \ra : \cN_{d} \times \cN_{d} \longrightarrow \cN_{d}, &  a \ra  b := Nor(a^b)
  \in \cN_{d}, & \forall a, b \in \cN_{d}.\\
 \end{array}
\end{equation}
It follows from Lemma \ref{lem:Daction}  (1)  that the two actions are
well-defined.

Define the map
\begin{equation}
  \label{eq:rho}
\rho_d: \cN_{d} \times \cN_{d} \longrightarrow \cN_{d} \times \cN_{d},
 \quad  \rho_d(a,b) := (a \la  b, a \ra  b).
\end{equation}
For simplicity of notation (when there is no ambiguity) we shall often write
$(\cN_{d}, \rho)$, where $\rho=\rho_d.$
\end{defnotation}
\begin{dfn}
\label{def:normalizedSol}
We call  $(\cN_d,\rho_d)$ \emph{the normalised $d$-Veronese solution associated with
$(X,r)$}.
\end{dfn}

\begin{pro}
\label{dVeron_monoid2}
In assumption and notation as above.
\begin{enumerate}
\item Let $\rho_d: \cN_d\times \cN_d \longrightarrow \cN_d\times \cN_d$ be the map defined as
$\rho_d (a, b) = (a \la  b, a \ra  b)$.  Then $(\cN_{d}, \rho_d)$ is a
solution of YBE of order $|\cN_{d}|=\binom{n+d-1}{d}=N$.
\item $(\cN_{d}, \rho_d)$  and $(S_d, r_d)$ are
    isomorphic solutions of YBE.
\end{enumerate}
\end{pro}

\begin{proof}
(1) By Corollary \ref{dVeron_monoid} $(S_d, r_d)$ is a nondegenerate symmetric
set, that is a solution of YBE. Thus by Remark \ref{rmk:YBE1} the left and the right actions associated
with $(S_d, r_d)$ satisfy conditions \textbf{l1}, \textbf{r1}, \textbf{lr3}, and
\textbf{inv}.
 Consider the actions $\la$ and $\ra$ on  $\cN_{d}$, given in
 Definition-Notation \ref{def:rho}. It follows from  (\ref{eq:actions}) and
 Lemma \ref{lem:Daction} that these actions also satisfy \textbf{l1},
 \textbf{r1}, \textbf{lr3} and  \textbf{inv}.
Therefore, by Remark \ref{rmk:YBE1} again, $ \rho_d$ obeys YBE, and is
involutive, so $(\cN_{d}, \rho_d)$ is a symmetric set.
Moreover, the nondegeneracy of $(S_d, r_d)$  implies that $(\cN_{d}, \rho_d)$
is nondegenerate.
By Corollary \ref{eq:dimAd} there are equalities  $|\cN_{d}|=|S_d|=
\binom{n+d-1}{d}=N$.

(2) We shall prove that the map $\Nor : S_d \longrightarrow \cN_d, \quad u
\mapsto \Nor (u)$ is an isomorphism of solutions. It is clear that the map is
bijective.
We have to show that $\Nor$ is a homomorphism of solutions, that is
\begin{equation}
  \label{eq:isom}
(\Nor \times \Nor) \circ r_d =  \rho_d\circ (\Nor \times \Nor).
\end{equation}
Let $(u, v)\in S_d\times S_d$, then the equalities $u =\Nor(u)$ and $v =
\Nor(v)$ hold in $S_d$, so
\[
\Nor ({}^uv)= \Nor ({}^{\Nor (u)}{\Nor (v)}), \quad  \Nor (u^v) =
\Nor({\Nor(u)}^{\Nor (v)})\]
which together with (\ref{eq:actions}) imply
\[\begin{array}{ll}
(\Nor \times \Nor) \circ r_d (u,v) &=  \Nor \times \Nor ({}^uv, u^v) =
(\Nor({}^uv), \Nor(u^v))\\
                                                &=(\Nor(u)\la \Nor(v), \Nor(u)
                                                \ra \Nor(v)) = \rho_d( \Nor (u),
                                                \Nor (v)).
\end{array}
\]
This proves (\ref{eq:isom}).
\end{proof}

Recall that monomials in $\cN_d$ are ordered lexicographically, $\cN_d := \{ w_1 < w_2< \dots < w_N \}$
see
\ref{eq:deg-d-monomials}
and we shall use this order throughout the paper.
\begin{pro}
\label{dVeron_monoid3}
In assumption and notation as above. Let $(\cN_d, \rho_d)$ be the normalised $d$-Veronese solution, see Definition \ref{def:normalizedSol}.
Then the Yang-Baxter algebra $B=\cA(\textbf{k}, \cN_{d}, \rho_d)$ is generated by
the set $\cN_{d}$ and has $\binom{N}{2}$  quadratic defining relations given
below:
\begin{equation}
  \label{eq:rel_rho}
\begin{array}{ll}
\Re=  \{f_{ji}= &w_jw_i - w_{i^{\prime}}w_{j^{\prime}}  \mid 1 \leq i,j \leq n,\; \text{where }\ \\
                &\rho_d(w_j, w_i)=(w_{i^{\prime}},w_{j^{\prime}}), \; \text{and }\;
                      w_j > w_{i^{\prime}} \; \text{holds in}\;\asX\}.
\end{array}
\end{equation}
Moreover,
 \begin{enumerate}

    \item[(i)] There is a 1-to 1 correspondence between the set of relations $\Re$ and the set of non-trivial $\rho_d$-orbits in $\cN_{d}\times \cN_{d}$;
        \item[(ii)] for every pair $(a, b)\in (\cN_{d} \times \cN_{d})\setminus
     \Fcal(\cN_d, \rho_d)$ the monomial $ab$ occurs exactly once in $\Re$;
\item[(iii)] Every relation  $f_{ji}$ has leading monomial $\LM(f_{ji}) =
    w_jw_i$.
\end{enumerate}
\end{pro}

\begin{proof}
 For simplicity of notation we set $\rho_d= \rho$.
 It is clear that there is a one-to-one correspondence between the set of relations of the
 algebra $B$ and the set of nontrivial orbits of the map $\rho$.

By definition each nontrivial relation of the Yang-Baxter algebra $B$ corresponds to a nontrivial orbit of $\rho$, and vice versa.
Say
\[\Ocal =\{(w_j,  w_i), \rho(w_j,  w_i) = (w_{i^{\prime}}, w_{j^{\prime}})\}=
\{(w_{i^{\prime}}, w_{j^{\prime}}), \rho(w_{i^{\prime}}, w_{j^{\prime}}) =
(w_j,  w_i)\},\]
and without loss of generality we may assume that the relation is
\[w_jw_i - w_{i^{\prime}}w_{j^{\prime}},  \; \text{where}\;  w_jw_i >
w_{i^{\prime}}w_{j^{\prime}}.\]

By Lemma \ref{lem:Daction} (2) the equality $w_jw_i =
w_{i^{\prime}}w_{j^{\prime}}$ holds in $S$.
 The monoid $S=S(X,r)$ is cancellative, see Remark \ref{fact1}.
hence an assumption that
$w_j = w_{i^{\prime}}$ would imply $w_i = w_{j^{\prime}}$, a contradiction.
Therefore $w_j > w_{i^{\prime}}$.

 Conversely, if
 $\cdot \rho(w_j,  w_i) =
 w_{i^{\prime}}w_{j^{\prime}}$ and $w_j > w_{i^{\prime}}$  then   $f_{ji}$ is a (nontrivial)
 relation of the algebra $B =B(\textbf{k}, \cN_{d}, \rho_d)$.
Clearly,  $w_j > w_{i^{\prime}}$ implies $w_jw_i >
 w_{i^{\prime}}w_{j^{\prime}}$ in $\asX$, so  $\LM(f_{ji}) = w_jw_i$, and the
 number of relations $g_{ji}$ is exactly
 $\binom{N}{2}$.
 \end{proof}

\subsection{The $d$-Veronese $\cA^{(d)}$  presented in terms of generators and
relations}

In Convention \ref{rmk:conventionpreliminary1}  and Notation as above the following result describes the $d$-Veronese $\cA^{(d)}$
of the Yang-Baxter algebra $\cA$ in terms of one-generators and quadratic
relations.

\begin{thm}
    \label{thm:d-Veronese_relations}  Let $d\geq 2$ be an integer. Let $(X,r)$ be
    a finite solution of YBE, where $X= \{x_1, \cdots, x_n\}$, let $\cA= \cA(\textbf{k}, X, r)$
    be its Yang-Baxter algebra, and let $(\cN_d, \rho)$ be the normalised $d$-Veronese solution
    from Definition \ref{def:normalizedSol}, where $\cN_d= \{w_1, \cdots, w_N\}$ is the set of normal monomials of length $d$ ordered lexicographically.

The $d$-Veronese subalgebra $\cA^{(d)} \subseteq \cA$
is a quadratic algebra with $N=\binom{n+d-1}{d}$ one-generators, namely the set
$\cN_{d}$ of normal monomials  of length $d$,
subject to $N^2 -\binom{n+2d-1}{n-1}$ linearly independent quadratic  relations
$\cR$ described below.
 \begin{enumerate}
 \item
 The relations $\cR$ split into two disjoint subsets $\cR= \cR_{a}
 \bigcup\cR_{b}$, as follows.
 \begin{enumerate}
 \item
 \label{thm:d-Veronese_relations_Ra}
The set $\cR_{a}$ contains $\binom{N}{2}$  relations corresponding
to the non-trivial $\rho$-orbits:
 \begin{equation}
  \label{eq:fji}
\begin{array}{ll}
\cR_{a}=  \{f_{ji}= &w_jw_i - w_{i^{\prime}}w_{j^{\prime}}  \mid 1 \leq i,j \leq n,\; \text{where } \\
                &\rho(w_j, w_i)=(w_{i^{\prime}},w_{j^{\prime}}), \; \text{and }\;
                      w_j > w_{i^{\prime}} \; \text{holds in}\;\asX\}.
\end{array}
\end{equation}
Each monomial $w_iw_j$, such that $(w_i, w_j)$ is in a nontrivial
$\rho$-orbit occurs exactly once in  $\cR_{a}$. Every relation  $f_{ji}$ has leading monomial $\LM(f_{ji}) =
    w_jw_i$.
    \item
 \label{thm:d-Veronese_relations_Rb}
The set $\cR_{b}$ contains $\binom{N+1}{2} -\binom{n+2d-1}{n-1}$
relations
 \begin{equation}
  \label{eq:gij}
  \begin{array}{ll}
\cR_{b} = \{g_{ij} = &w_i w_j -w_{i_0}w_{j_0}\mid 1 \leq i,j \leq n,\; \text{where }\; \cdot\rho(w_i, w_j)\geq w_iw_j, w_{i_0}, w_{j_0} \in\cN_d, \\
                     &\text{and } w_iw_j > \Nor(w_iw_j) = w_{i_0}w_{j_0} \in \cN_{2d}\; \text{is the normal form of }\; w_{i}w_{j}\}
\end{array}
\end{equation}
 In particular,
 $\LM(g_{ij}) = w_iw_j >  w_{i_0}w_{j_0}.$
\end{enumerate}

 \item
 The d-Veronese subalgebra $\cA^{(d)}$  has a second set
 of linearly independent quadratic relations, $\cR_{1}$, which splits into two
 disjoint subsets $\cR_{1}= \cR_{1a} \bigcup\cR_{b}$ as follows.

\begin{enumerate}
 \item
 \label{thm:d-Veronese_relations1a}
 The set $\cR_{1a}$ is \emph{a reduced version} of $\cR_{a}$ and contains
 exactly $\binom{N}{2}$  relations
 \begin{equation}
  \label{eq:gji}
  \begin{array}{ll}
\cR_{1a} = \{ g_{ji} &= w_jw_i -w_{i^{\prime\prime}}w_{j^{\prime\prime}}
          \mid 1 \leq i,j \leq n,\; \text{where } \\
                &\cdot\rho(w_j, w_i)< w_jw_i
                \; \text{and }\; w_{i^{\prime\prime}}w_{j^{\prime\prime}}=\Nor(w_jw_i)\in \cN_{2d}
                     \}.
                     \end{array}
\end{equation}
 In particular, $LM(g_{ji}) = w_jw_i > w_{i^{\prime\prime}}w_{j^{\prime\prime}}\in
 \cN_{2d}.$
\item
\label{thm:d-Veronese_relations1b}
The set $\cR_{b}$ is given in
(\ref{eq:gij}).
\end{enumerate}

\item  The two sets of relations $\cR$ and $\cR_{1}$  are equivalent:
$\cR  \Longleftrightarrow  \cR_1$.
\end{enumerate}
    \end{thm}

\begin{proof}
We start with a general observation.
By Convention
\ref{rmk:conventionpreliminary1}  we identify the algebra  $\cA$ with
$(\textbf{k}\cN,   \bullet ).$
We know that the $d$-Veronese subalgebra $\cA^{(d)}$ is one-generated and
quadratic, see Corollary \ref{cor:dVeron}.
By (\ref{eq:A^d})
\[
\cA^{(d)} = \bigoplus_{m\in\N_0}  \cA_{md} \cong \bigoplus_{m\in\N_0}
\textbf{k}\cN_{md}.
\]
So $\cA^{(d)}_1 =\textbf{k}\cN_{d}$ and the ordered monomials
$w \in \cN_d$ of length $d$ are degree one generators of $\cA^{(d)}$. There are
equalities
\[
\dim \cA_d = |\cN_d |=\binom{n+d-1}{d} =N.
\]

Moreover,
\[
\dim (\cA^{(d)})_2 = \dim(\cA_{2d}) = \dim (\textbf{k}\cN_{2d}) = |\cN_{2d}|=\binom{n+2d-1}{n-1}.
\]

We want to find  a finite  presentation in terms of generators and relations
\[\cA^{(d)} = \textbf{k}\langle w_1, \cdots, w_N\rangle/ (R),\]
where the two-sided (graded) ideal $I= (R)$ is generated by linearly independent homogeneous relations $R$ of degree $2$ in the variables $w_i$ so
$I_2 = \Span_{\textbf{k}} R.$
\textbf{1.} We compare dimensions
to find the number of quadratic linearly independent relations for the
$d$-Veronese $\cA^{(d)}$.
The equality of vectors spaces \[\textbf{k}\langle w_1, \cdots, w_N\rangle = I \oplus (\bigoplus_{m\in\N_0}
\textbf{k}\cN_{md})\]
implies an equality for the graded components
\[
(\textbf{k}\langle w_1, \cdots, w_N\rangle)_2 = I_2 \oplus \textbf{k}\cN_{2d}
\]
Hence if $R$ is a set of linearly independent quadratic relations  defining
$\cA^{(d)}$, that is $R$ generates the ideal of relations $I= (R)$ there must be an equality
$|R|+\dim \cA^{(d)}_2= N^2$, so
\begin{equation}
  \label{eq:card_R}
|R|= N^2-  \binom{n+2d-1}{n-1}.
\end{equation}
We shall prove that the set of quadratic polynomials $\cR= \cR_{a}
\bigcup\cR_{b}$ given above consists of relations of $\cA^{(d)}$, it has order
$|\cR| =  N^2-  \binom{n+2d-1}{n-1},$ and is linearly independent.

\begin{enumerate}
 \item[(a)]  Observe that there is a 1-to-1 correspondence between the polynomials $f_{ji}\in \cR_{a}$ and the set on nontrivial $\rho$-orbits in
 $\cN_{d}\times \cN_{d}$, and therefore $\cR_a$ has exactly $\binom{N}{2}$ elements.
 The $\rho$-orbits in $\cN_d\times\cN_d$   are disjoint and therefore every monomial $w_iw_j$, with $1 \leq i,j \leq N$, such that $(w_i, w_j)$ is in a nontrivial $\rho$-orbit occurs exactly once in some polynomial $f \in \cR_a$.
 We claim that $\cR_a$ consists of relations of $\cA^{(d)}$.
Consider an element $f_{ji}=w_jw_i - w_{i^{\prime}}w_{j^{\prime}} \in  \cR_a $.  It is obvious that $f_{ji}$ is not identically zero in  $\textbf{k}\langle w_1, \cdots, w_N\rangle$.
We have to show that $w_jw_i-w_{i^{\prime}}w_{j^{\prime}}=0$, or
equivalently,  $w_jw_i=w_{i^{\prime}}w_{j^{\prime}}$ holds
in $\cA^{(d)}$. But $\cA^{(d)}$ is a subalgebra of the Yang-Baxter algebra $\cA$ which is isomorphic to the monoid algebra $\textbf{k}S$.
Thus it  will be enough to prove that
\begin{equation}
  \label{eq:rel1}
w_jw_i= w_{i^{\prime}}w_{j^{\prime}}\quad \text{is an equality in $S$}.
\end{equation}
Note that $\cN$ is a subset of $\langle X\rangle$ and $a=b$ in $\cN$ is
equivalent to $a,b\in \cN$ and $a=b$ as words in $\langle X\rangle$. Clearly,
each equality of words in $\langle X\rangle$ holds also in $S$.

By assumption
\begin{equation}
  \label{eq:rel2}
 \rho(w_j,w_i) =(w_{i^{\prime}}, w_{j^{\prime}})\quad \text{holds in
 }\cN_d\times\cN_d.
 \end{equation}
By Definition-Notation \ref{def:rho}, see  (\ref{eq:actions}) and (\ref{eq:rho})
one has
\begin{equation}
  \label{eq:rel3}
 \rho(w_j,w_i) = (\Nor({}^{w_j}{w_i}),\Nor(w_j^{w_i})),\;\text{in}\;
 \cN_d\times\cN_d
\end{equation}
and comparing (\ref{eq:rel2}) with (\ref{eq:rel3}) we obtain that
\begin{equation}
  \label{eq:rel4}
 \Nor({}^{w_j}{w_i})=w_{i^{\prime}},\;\text{and}\;
 \Nor(w_j^{w_i})=w_{j^{\prime}}\; \text{are equalities of words
 in}\;\cN_d\subset X^d.
\end{equation}

The equality $u = \Nor(u)$ (the normal form of $u$, modulo the ideal $I$) holds in $S$ and in $\cA$, for every $u \in \langle
X\rangle$, therefore the following are equalities in $S$:
\begin{equation}
  \label{eq:rel5}
  \begin{array}{lcl}
 \Nor({}^{w_j}{w_i})= {}^{w_j}{w_i},&& \Nor(w_j^{w_i})=w_j^{w_i} \\
 (\Nor({}^{w_j}{w_i}))(\Nor(w_j^{w_i}))&=& ({}^{w_j}{w_i})(w_j^{w_i}).
 \end{array}
\end{equation}
Now
(\ref{eq:rel4}) and (\ref{eq:rel5}) imply that
  \begin{equation}
  \label{eq:rel6}
  w_{i^{\prime}}w_{j^{\prime}}=({}^{w_j}{w_i})(w_j^{w_i}) \quad \text{holds
  in}\; S.
 \end{equation}
But $S$ is an M3- braided monoid, so by condition  (\ref{eq:braided_monoid}) M3,
the following is an equality in $S:$
\begin{equation}
  \label{eq:rel7}
  w_jw_i =({}^{w_j}{w_i})(w_j^{w_i}).
 \end{equation}
This together with (\ref{eq:rel6}) imply the desired equality
$
w_jw_i= w_{i^{\prime}}w_{j^{\prime}}$ in $S$.
It follows that $f_{ji}= w_jw_i- w_{i^{\prime}}w_{j^{\prime}}$ is identically
$0$ in
$\cA$ and therefore in $\cA^{(d)}$.

Clearly, for $f_{ji}= w_jw_i- w_{i^{\prime}}w_{j^{\prime}}$, the  inequality $w_j > w_{i^{\prime}}$ implies that  $w_jw_i>w_{i^{\prime}}w_{j^{\prime}}$ as elements of $\asX$, so
 the leading monomial of every relation $f_{ji} \in \cR_a$ is
$\LM(f_{ji}) = w_jw_i$.

\item[(b)]
Next we consider the elements  $g_{ij} = w_i w_j -w_{i_0}w_{j_0}\in \cR_{b}$. These are homogeneous polynomials of degree
$2d$. It follows from their description that $ w_iw_j >\Nor(w_iw_j)=w_{i_0}w_{j_0}$, so
their leading monomials satisfy $\LM(g_{ij}) = w_iw_j$.

Moreover, the description of $\cR_{b}$ implies that there is a 1-to-1 correspondence between the polynomials in $\cR_{b}$ and the set of all $\rho$- orbits $\cO$ \emph{which do not contain elements} $(a,b)\in \cN_d\times \cN_d$ such that $ab$ is in normal form,  that is $ab \in \cN_{2d}$.

If $\cO$ is such an orbit, then  $\cO = \{(w_i, w_j), \rho(w_i, w_j) \}$ where  $\cdot \rho(w_i, w_j)\geq w_iw_j$ and $w_iw_j$ is not in normal form. (In particular, $\cO$ can be a one-elemet orbit.) Clearly,
$w_i w_j= \Nor(w_iw_j)$ is an
identity in $\cA$, (and in $(\textbf{k}\cN,  \bullet )$). The normal form
$\Nor(w_iw_j)$ is a
monomial of length $2d$, so it can be written as
a product $\Nor(w_iw_j)= w_{i_0}w_{j_0},$ where $w_{i_0}, w_{j_0} \in \cN_d$,
It follows that
\[g_{ij} = w_i w_j -w_{i_0}w_{j_0}= 0
\]
is an identity in $\cA$ (and in $(\textbf{k}\cN,  \bullet )$).
Conversely, it follows from the description of $\cR_b$ that each relation $g_{ij} = w_i w_j -w_{i_0}w_{j_0}\in \cR_b$ determines uniquely the $\rho$- orbit $\cO = \{(w_i, w_j), \rho(w_i, w_j) \}$ with the above properties.
Note that
each $(a,b)\in \cN_d\times \cN_d$ such that  $ab \in \cN_{2d}$ belongs to exactly one orbit so the number of such orbits equals the cardinality
\[|\cN_{2d}|= \binom{n+2d-1}{n-1}.\]

The number of all $\rho$- orbits in $\cN_d\times \cN_d$ (including the one-element orbits) is
$\binom{N+1}{2}$. Therefore the number of disjoint $\rho$- orbits which "produce" distinct leading monomials $w_iw_j$
for the (distinct) elements $g_{ij} = w_i w_j -w_{i_0}w_{j_0}\in \cR_{b}$ is exactly
\begin{equation}
  \label{eq:card_Rb}
\binom{N+1}{2} - \binom{n+2d-1}{n-1}= |\cR_{b}|.
\end{equation}
\end{enumerate}
The sets $\cR_{a}$ and $\cR_{b}$ are disjoint, since $\{\LM(f)\mid f \in
\cR_{a}\} \cap \{\LM(g)\mid g \in \cR_{b}\}= \emptyset.$
Therefore there are equalities:
\begin{equation}
\label{eq: card_cR}
|\cR|= |\cR_{a}| +|\cR_{b}| =\binom{N}{2} +(\binom{N+1}{2} -
\binom{n+2d-1}{n-1})= N^2 - \binom{n+2d-1}{n-1},
\end{equation}
so the set $\cR$ has exactly the desired number of relations  given in
(\ref{eq:card_R}).
It remains to show that
 $\cR$ consits of linearly independent elements
of $\textbf{k}\asX$.

\begin{lem}
    \label{lem:independence}
Under the hypothesis of Theorem \ref{thm:d-Veronese_relations},   the set of
polynomials $\cR \subset \textbf{k}\asX$ is linearly independent.
\end{lem}

\begin{proof}
It is well known that the set of all words in $\asX$ forms a basis of
$\textbf{k}\asX$ (considered as a vector space), in particular every finite set
of distinct words in $\asX$ is linearly independent. All words occurring in
$\cR$ are elements of $X^{2d}$, but some of them occur in more than one
relation, e.g. every $w_iw_j$, which is not a fixed point of $\rho$ but is the leading monomial of $g_{ij}\in \cR_b$, occurs also as a second term
of a polynomial $f_{pq}= w_pw_q - w_iw_j \in \cR_{a}$, where $\cdot \rho(w_i,w_j)
= w_pw_q>w_iw_j$
We shall prove the lemma in three steps.

\begin{enumerate}
\item
  The set of polynomials $\cR_{a}\subset \textbf{k}\asX$ is linearly
  independent.

Notice that the polynomials in $\cR_a$ are in 1-to-1 correspondence with the
nontrivial $\rho$-oprbits in $\cN_{d}\times \cN_{d}$:
Each polynomial $f_{ji} = w_jw_i -w_{i^{\prime}}w_{j^{\prime}}\in \cR_{a}$ is formed out
of the two monomials in the nontrivial $\rho$-oprbit
$\{(w_j, w_i) , (w_{i^{\prime}}, w_{j^{\prime}})= \rho(w_j, w_i)\}$. But the
$\rho$-orbits are disjoint, hence each monomial $\cdot(a,b)$,
with $(a,b) \neq \rho(a,b)$  occurs exactly once in $\cR_a.$
Present each $f\in \cR_a$ as $f = u_f -v_f,$ where $u_f = \LM(f)$. Then a linear relation
\[
0= \sum_{f \in \cR_a}\alpha_f f = \sum_{f \in\cR_a} \alpha_f (u_f -v_f) \text{where
all}\; \alpha_f \in \textbf{k}
 \]
involves only pairwise distinct monomials in $X^{2d}$ and therefore it must be
trivial: $\alpha_f= 0, \forall  f \in \cR_a $.
It follows that $\cR_a$ is linearly independent.
\item
The set  $\cR_{b} \subset \textbf{k}\asX$ is linearly independent. Assume the contrary.
Present each elements of $\cR_b$ as $g = u_g-v_g\in \cR_b$, where $u_g - \LM(g)$.
  Then there exists a nontrivial linear relation for the elements of $\cR_b:$
\begin{equation}
\label{eq:Rb1}
\sum_{g \in \cR_b} \beta_g g = \sum_{g \in \cR_b} \beta_g  (u_g-v_g) = 0
 \; \text{with}\; \beta_g
\in \textbf{k}.
\end{equation}
Let $g_{ij}$ be the polynomial with $\beta_{ij}= \beta_{g_{ij}}\neq 0$ whose leading monomial is the highest among all leading monomials of polynomials
$g\in \cR_b$, with $\beta_{g}\neq 0$.
So we have
\[
\LM(g_{ij}) = w_iw_j > \LM(g), \; \text{for all}\; g \in \cR_b,  g \neq g_{ij}\; \text{where}\; \beta _{g} \neq 0.
\]

We use (\ref{eq:Rb1})  to yield the following equality in $\textbf{k}\langle X \rangle$:
\[
w_iw_j =   w_{i_0}w_{j_0} -
\sum_{g \in \cR_b, \LM(g)< w_iw_j} \; \frac{\beta_{g}}{\beta_{ij}}\;(u_g-v_g).
\]

Observe that the right-hand side of this equality  is a linear combination of monomials strictly smaller than $w_iw_j$, which is impossible.
It follows that the set $\cR_{b} \subset \textbf{k}\asX$ is linearly independent.

\item The set  $\cR \subset \textbf{k}\asX$ is linearly independent.
For simplicity of notation (as before) we present every $f\in \cR_a$ and every $g\in \cR_b$ as
\[f= u_f-v_f, \text{where} \; u_f=\LM(f)> v_f,\; g= u_g-v_g, \text{where} \; u_g=\LM(g)> v_g.  \]

Assume the polynomials in $\cR$ satisfy a linear relation
\begin{equation}
  \label{eq:linindept1}
\sum_{f \in \cR_a}\alpha_f f +\sum_{g \in \cR_b}\beta_g g = 0, \;\text{where for all} f\in \cR_a, g \in \cR_b\;\; \alpha_f,
\beta_g \in \textbf{k}.
\end{equation}
This gives the following equality in the free associative algebra
$\textbf{k}\asX$:
\begin{equation}
  \label{eq:linindept2}
S_1 =\sum_{f \in \cR_a} \alpha_{f} u_f =
\sum_{f \in \cR_a} \alpha_{f} v_f -\sum_{g\in \cR_b}\beta_g g=S_2.
\end{equation}
The element $S_1 =\sum_{f \in \cR_a} \alpha_{f} u_f$
is in the space
$U=\Span B_1$,
where $B_1=\{\LM(f) \mid f \in \cR_a\}$ is linearly
independent.

The element \[S_2=\sum_{f \in \cR_a} \alpha_{f} v_f -\sum_{g\in \cR_b}\beta_g g\]
on the right-hand side of the equality is in the space $V=\Span B,$
where
\[
B=\{v_f\mid f \in \cR_a\}\cup
\{u_g, v_g\mid g \in \cR_b\}.
\]

Take a subset $B_2 \subset B$ which forms a basis of $V$. Note that
$B_1\cap B = \emptyset,$ hence $B_1\cap B_2=\emptyset.$
Moreover, each of the sets $B_1$, and $B_2$ consists of pairwise distinct and therefore linearly independ
monomials and it is easy to show that
$U\cap V = 0.$  Thus the equality $S_1 = S_2\in U\cap V = 0$ implies a linear
relation
\[S_1 =\sum_{f \in \cR_a} \alpha_{f} u_f =0
\]
for the set of leading monomials $u_f = \LM(f), f \in \cR_a$  which are pairwise distinct, and
therefore linearly independent. It follows that   $\alpha_f= 0$, for all $f \in \cR_a$.
 This together with (\ref{eq:linindept1}) implies the
linear relation
\[
\sum_{g \in \cR_b}\beta_g g = 0
\]
 and since by (2) $\cR_b$ is linearly independent we get again $\beta_g= 0,
 \forall  \; g \in \cR_b$.
It follows that
the linear relation (\ref{eq:linindept1}) must be trivial, and therefore
$\cR$ is a linearly independent set of polynomials.
\end{enumerate}
\end{proof}

It is now easy to see that  $\cR$ is a set of defining relation for the $d$- Veronese subalgebra $\cA^{(d)}$.

We know that $\cA^{(d)}$ is a quadratic algebra whose one-generators are the monomials $w_1, \cdots, w_N$, that is its ideal of relations $I$ is generated by homogeneous polynomials of degree 2 in the $w_i$'s.  Consider the ideal $J = (\cR)$ of the free associative algebra $\textbf{k}\langle w_1, \cdots, w_N\rangle$. We have proven that each element of $\cR$ is a relation of $\cA^{(d)}$, therfore $J\subseteq I$. To show that $J$ is the ideal of relations of $\cA^{(d)}$ it will be enough to verify thst there is an isomorphism of vector spaces:
\[(\cR)_2\oplus (\cA^{(d)})_2= (\textbf{k}\langle w_1, \cdots, w_N\rangle)_2,\]
or equivalently,
\[\dim \Span_{\textbf{k}}\cR + \dim (\cA^{(d)})_2 = \dim (\textbf{k}\langle w_1, \cdots, w_N\rangle)_2.\]
We have shown that $\cR$ is linearly independent, so
$\dim \Span_{\textbf{k}}\cR =|\cR|= N^2 - \binom{n+2d-1}{n-1}$, see \ref{eq: card_cR}. On the other hand
$\dim (\cA^{(d)})_2= \dim \cA_{2d}= \binom{n+2d-1}{n-1}$, see Corollary \ref{cor:dimAd}.
Therefore
\[\dim \Span_{\textbf{k}}\cR + \dim (\cA^{(d)})_2 = N^2 - \binom{n+2d-1}{n-1}+ \binom{n+2d-1}{n-1} = N^2 = \dim (\textbf{k}\langle w_1, \cdots, w_N\rangle)_2, \]
as desired. Therfore the set $\cR$ is a set of defining relations for the Veronese subalgebra $\cA^{(d)}$.
We have proven part (1) of the theorem.

Analogous argument proves part (2). Note that the polynomials of  $\cR_{1a}$ are
reduced from $\cR_{a}$ using $\cR_{b}$. It is not difficult to prove the equivalence $\cR  \Longleftrightarrow  \cR_1$.
 \end{proof}

\section{Veronese maps}
\label{sec:Veronesemap}

In this section we shall introduce an analogue of Veronese maps between quantum spaces (Yang-Baxter
algebras) associated to finite  solutions of YBE. We keep the notation and
all conventions from the previous sections. As usual, $(X,r)$ is a finite
solution of order $n$, $\cA= \cA(\textbf{k}, X, r)$ is the associated Yang-Baxter algebra,
where we fix an enumeration,
$X= \{x_1, \cdots, x_n\}$
  as in Convention \ref{rmk:conventionpreliminary1}, $d\geq 2$ is an integer,
  $N=\binom{n+d-1}{d}$, and
$\cN_d = \{w_1 < w_2 < \cdots < w_N\}$ is  the set of all normal monomials of length $d$ in
$X^d$  ordered lexicographically,   as in (\ref{eq:deg-d-monomials}).

\subsection{The $d$-Veronese solution of YBE associated to a finite solution
$(X,r)$}
\label{subsec:dveronesesolution}

We have shown that the braided monoid $(S, r_S)$ associated to $(X,r)$ induces
 the normalised $d$-Veronese solution $(\cN_{d}, \rho_d)$ of order $N=\binom{n+d-1}{d}$, see Definition \ref{def:normalizedSol}.
We shall use this construction to introduce the notion of \emph{a $d$-Veronese
solution of YBE associated to $(X,r)$}, denoted by $(Y, r_Y)$.

\begin{defnotation}
\label{def:Yrho}
In notation as above.
Let $(X,r)$ be a finite  solution,   $X = \{x_1, \cdots, x_n\}$, let $\cN_d =
\{w_1 < w_2 < \cdots w_N\}$ be the set of normal monomials of length $d$, and let $(\cN_d, \rho)=  (\cN_d, \rho_d)$ be the
normalised $d$-Veronese solution.
Let
$Y = \{y_1, y_2, \cdots, y_N\}$
be an abstract set and consider  the quadratic set $(Y, r_Y)$ ,  where
the map $r_Y:Y\times Y\longrightarrow Y\times Y$ is defined as
\begin{equation}
  \label{eq:rY}
r_Y (y_j, y_i) :=(y_{i^{\prime}},y_{j^{\prime}})\quad\text{\emph{iff}}\quad \rho
(w_j, w_i)= (w_{i^{\prime}},w_{j^{\prime}}), \; 1 \leq i,j,i^{\prime},
j^{\prime}\leq n.
\end{equation}

It is straightforward that  $(Y, r_Y)$ is a solution of YBE (a nondegenerate symmetric set) of order
$N$ isomorphic to $(\cN_d, \rho_d)$. We shall refer to it as \emph{the
$d$-Veronese solution of YBE associated to  $(X,r)$.}
\end{defnotation}

By Corollary \ref{cor:orbits_and_fixedpoints}  the set $Y\times Y$ splits into
$\binom{N}{2}$  two-element $r_Y$- orbits and $N$ one-element $r_Y$-orbits.

As usual, we consider the degree-lexicographic ordering on the free monoid $\asY$
extending $y_1<y_2<  \cdots < y_N$.
The Yang-Baxter algebra
$\cA_Y= \cA(\textbf{k},Y,r_Y)
\simeq \textbf{k}\langle Y ; \;\cR_{Y}
\rangle $
has exactly $\binom{N}{2}$ quadratic relations which can be written
explicitly as
\begin{equation}
  \label{eq:gammaji}
\cR_{Y} = \{ \gamma_{ji} = y_jy_i -y_{i^{\prime}}y_{j^{\prime}}\mid \; 1 \leq i,j\leq N, r_Y(y_j, y_i)= (y_{i^{\prime}},y_{j^{\prime}}), \;\text{where}\; y_jy_i > y_{i^{\prime}}y_{j^{\prime}} \;\text{holds in }\;\asY
\},
\end{equation}
Each relation
corresponds to a non-trivial $r_Y$-orbit.
 The leading monomials satisfy
 $\LM(\gamma_{ji}) = y_jy_i > y_{i^{\prime}}y_{j^{\prime}}$.

\subsection{The Veronese map $v_{n,d}$ and  its kernel}
\label{subsec:grobner_veronese_map}
\begin{lem}
    \label{lem:Ver_well-defined1}
In notation as above. Let $(X,r)$ be a solution of order $n$, $\cA_X=
\cA(\textbf{k},X,r)$, let $d\geq 2$, be an integer, and let
$N=\binom{n+d-1}{d}$. Suppose $(Y, r_Y)$ is the associated $d$-Veronese
solution,  $Y=\{y_1, \cdots, y_N\}$, and $\cA_Y = \cA(\textbf{k},Y,r_Y)$, is the
corresponding Yang-Baxter algebra.

The assignment
    \[y_1 \mapsto w_1, \;  \; y_2 \mapsto w_2,  \; \;   \dots ,  \; \;   y_N
    \mapsto w_N\]
extends to an algebra  homomorphism  $v_{n,d}:\cA_Y  \rightarrow \cA_X.$
The image of the map $v_{n,d}$ is the d-Veronese subalgebra $\cA_X^{(d)}$.
    \end{lem}
    \begin{proof}
    Naturally we set
    $v_{n,d} (y_{i_{1}}\cdots y_{i_{p}}): = w_{i_{1}}\cdots w_{i_{p}}$, for all
    words $y_{i_{1}}\cdots y_{i_{p}}\in \langle Y\rangle$ and then extend this
    map linearly.
    Note that for each polynomial $\gamma_{ji}\in \cR_{Y}$ one has
    \[v_{n,d}( \gamma_{ji}) = f_{ji}\in \cR_{a},\]
    where the set $\cR_{a}$ is a  part of the relations of $\cA_X^{(d)}$ given
    in (\ref{eq:fji}).
    Indeed, let $\gamma_{ji}\in \cR_{Y}$, so
    $\gamma_{ji} = y_jy_i -y_{i^{\prime}}y_{j^{\prime}},$ where
    $(y_{i^{\prime}},y_{j^{\prime}}) = r_Y(y_j, y_i)$, and $y_jy_i >y_{i^{\prime}}y_{j^{\prime}}$,  see  (\ref{eq:rY}).
    Then
    \[
    \begin{array}{ll}
    v_{n,d}( \gamma_{ji}) &= v_{n,d}(y_jy_i -y_{i^{\prime}}y_{j^{\prime}}) \\
                          &= w_jw_i
                          -w_{i^{\prime}}w_{j^{\prime}},\;\text{where}\;
                          (w_{i^{\prime}},w_{j^{\prime}}) = \rho (w_j, w_i), \text{and}\; w_jw_i
                          >w_{i^{\prime}}w_{j^{\prime}}\;
                          \\
                          & = f_{ji}\in \cR_{a}.
    \end{array}
    \]
   We have shown that $f_{ji}$ equals identically $0$ in $\cA_X$, so the map $v_{n,d}$ agrees
   with the relations of the algebra $\cA_X$.
It follows that $v_{n,d}:\cA_Y  \rightarrow \cA_X$
is a well-defined homomorphism of algebras.

The image of $v_{n,d}$ is the subalgebra of $\cA_X$ generated by the normal
monomials $\cN_d$, which by Theorem \ref{thm:d-Veronese_relations} is exactly
the $d$-Veronese subalgebra
$\cA_X^{(d)}$.
    \end{proof}
\begin{dfn}
\label{def:veronesemap}
We call the map $v_{n,d}$ from Lemma \ref{lem:Ver_well-defined1} \textit{the
$(n,d)$-Veronese map}.
\end{dfn}

\begin{thm}
    \label{thm:Veronese_ker}
In assumption and notation as above.
Let $(X,r)$ be a solution of order $n$, with $X=\{x_1, \cdots, x_n\}$,
let $\cA_X= \cA(\textbf{k},X, r)$ be  its Yang-Baxter algebra. Let $d \geq 2$
be an integer, $N=\binom{n+d-1}{d},$ and suppose that $(Y, r_Y)$ is the associated
$d$-Veronese solution of YBE with  corresponding Yang-Baxter algebra  $\cA_Y = \cA(\textbf{k},Y,r_Y)$,
Let $v_{n,d} : \cA_Y \rightarrow  \cA_X$
be the the
$(n,d)$-Veronese map (homomorphism of algebras) extending the assignment
\[y_1
\mapsto w_1, \ y_2 \mapsto w_2, \ \dotsc ,\  y_N  \mapsto w_N.\]
Then the following conditions hold.
\begin{enumerate}
\item
The image of $v_{n,d}$ is the d-Veronese subalgebra $\cA_X^{(d)}$ of  $\cA_X$.
\item
The kernel  $\mathfrak{K} := \ker (v_{n,d})$  of the Veronese map is generated
by the set
of $\binom{N+1}{2}- \binom{n+2d-1}{n-1}$ linearly independent quadratic
binomials:
\begin{equation}
  \label{eq:gammaij}
\cR (\mathfrak{K}) = \{\gamma_{ij} = y_i y_j -y_{i_0}y_{j_0}\mid 1\leq i,j\leq n, \; \text{where }\; g_{ij} = w_i w_j -w_{i_0}w_{j_0}\in \cR_{b} \}
\end{equation}
 In particular, the leading monomial of each $\gamma_{ij}$ satisfies
 \[\LM(\gamma_{ij}) = y_iy_j >  y_{i_0}y_{j_0}.\]
\end{enumerate}
\end{thm}

\begin{proof}
Part (1) follows from Lemma \ref{lem:Ver_well-defined1}.

Part (2). We have to verify that the set $\cR (\mathfrak{K})$ generates $\mathfrak{K}$.

By direct computation one shows that  for every $\gamma_{ij} \in \cR (\mathfrak{K})$ one has
\[v_{n,d} (\gamma_{ij})(y_1, \cdots, y_N)= g_{ij}(w_1, \cdots w_n) \in \cR_b ,\]
in fact
$v_{n,d}$ induces a 1-to-1 map $\cR (\mathfrak{K}) \longrightarrow \cR_b$. It follows that
\begin{equation}
\label{eq: card_R(K)}
|\cR (\mathfrak{K})|= |\cR_b|= \binom{N+1}{2} -\binom{n+2d-1}{n-1}.
\end{equation}
Moreover, $v_{n,d}(\cR (\mathfrak{K})) = \cR_b,$ the set
of relations of the d-Veronese $\cA_X^{(d)}$ given in (\ref{eq:gij}),
so $\cR (\mathfrak{K})\subset \mathfrak{K}$.

The Yang-Baxter algebra $ \cA_Y$ is a quadratic algebra with $N$ generators and
$\binom{N}{2}$ defining quadratic relations which are linearly independent, so
\[\dim (\cA_Y)_2=N^2-\binom{N}{2}=\binom{N+1}{2}.\]
By the First Isomorphism Theorem  $(\cA_Y/\mathfrak{K})_2\cong (\cA_X^{(d)})_2
=(\cA_X)_{2d}$, hence
\[\dim (\cA_Y)_2 = \dim (\mathfrak{K})_2 + \dim (\cA_X)_{2d}.\]
We know that  $\dim (\cA_X)_{2d}=|\cN_{2d}|= \binom{n+2d-1}{n-1}$, hence
\[
\binom{N+1}{2}= \dim (\mathfrak{K})_2 + \binom{n+2d-1}{n-1}.
\]
This together with (\ref{eq: card_R(K)}) implies that
\[\dim (\mathfrak{K})_2 = \binom{N+1}{2} - \binom{n+2d-1}{n-1} = |\cR (\mathfrak{K})|.
\]

The set $\cR (\mathfrak{K})$ is linearly independent, since  $v_{n,d}(\cR (\mathfrak{K})) =
\cR_b,$  and by Lemma \ref{lem:independence} the set $\cR_b$ is linearly
independent. This together with the equality $|\cR (\mathfrak{K})|=\dim (\mathfrak{K})_2$ implies that the set $\cR (\mathfrak{K})$
is a basis of the graded component $\mathfrak{K}_2$,
so $\mathfrak{K}_2 = \textbf{k} \cR (\mathfrak{K})$. But the ideal $\mathfrak{K}$ is
generated by homogeneous polynomials of degree $2$, and
therefore
\begin{equation}
\label{eq:K-generators}
\mathfrak{K} = (\mathfrak{K}_2)= (\cR (\mathfrak{K})).
\end{equation}
We have proven that $\cR (\mathfrak{K})$ is a minimal set of generators for the kernel
$\mathfrak{K}$.
\end{proof}

\begin{cor}
\label{cor:dVeron2}
Let $(X,r)$ be a solution of order $n$, and let $\cA = \cA(\textbf{k}, X,r)$ be
its Yang-Baxter algebra, let $d \geq 2$ be an integer.
Then the $d$-Veronese subalgebra $\cA^{(d)}$ is a left and a right Noetherian domain.
\end{cor}
\begin{proof}
The $d$-Veronese $\cA^{(d)}$  is a subalgebra of  $\cA$ which is a domain,
see Remark \ref{fact1}  and therefore $\cA^{(d)}$ is a domain.
By Theorem  \ref{thm:Veronese_ker}  $A^{(d)}$ is a homomorphic image of
the Yang-Baxter algebra $\cA_Y= \cA(\textbf{k}, Y, r_Y)$, where  $(Y, r_Y)$ is the $d$-Veronese solution
associated with $(X,r)$. The algebra $\cA_Y$ is Noetherian, since
 $(Y, r_Y)$ is a finite solution of YBE, see Remark \ref{fact1}, so $\cA^{(d)}$ is a left and a right Noetherian domain.
\end{proof}
\section{Special cases}
\label{sec:square-free}
\subsection{Veronese subalgebras of the Yang-Baxter algebra of a square-free
solution}
\label{subsec:square-free}
In this subsection $(X,r)$ is a finite square-free solution of YBE of order $n$, $d
\geq 2$ is an integer. We keep the conventions and notation from the previous
sections.
We apply  Remark \ref{factAS} and fix an appropriare enumeration $X= \{x_1, \cdots, x_n\}$, such that the
the Yang-Baxter algebra $\cA = \cA(\textbf{k}, X,r)$ is a
binomial skew polynomial ring.
More precisely, $\cA$ is a PBW algebra
 $\cA=\textbf{k} \langle x_1, \cdots , x_n\rangle/(\Re_{\cA})$,
where
\begin{equation}
\label{eq:skewpol1}
\Re_{\cA}=\{\varphi_{ji} = x_{j}x_{i} -
x_{i^\prime}x_{j^\prime} \mid 1\leq i<j\leq n\},
\end{equation}
is such that for every pair $i, j, \; 1\leq
i<j\leq n$, the relation $\varphi_{ji}= x_{j}x_{i} - x_{i'}x_{j'}\in \Re_{\cA},$
satisfies
$j> i^{\prime}$, $i^{\prime} < j^{\prime}$ and every term $x_ix_j, 1\leq i <j \leq
n,$ occurs in some relation in $\Re_{\cA}$.
In particular
\begin{equation}
        \label{eq:LM}
\LM (\varphi_{ji}) = x_jx_i, \,    1 \leq i < j \leq n.
\end{equation}

The set $\Re_{\cA}$ is a quadratic Gr\"{o}bner basis of the ideal $I = (\Re_{\cA})$
w.r.t the degree-lexicographic ordering $<$ on $\asX$.
It follows from the shape of the elements of the Gr\"{o}bner basis $\Re_{\cA}$,  and
(\ref{eq:LM})
that the set $\cN=\cN(I)$ of normal monomials modulo $I =(\Re_{\cA})$ coincides with
the set $\cT$ of ordered monomials (terms) in $X$,
\begin{equation}
\label{eq:skewpol12}
\cN=\cT =\cT(X) =\left\lbrace x_1^{\alpha_1}\cdots x_n^{\alpha_n}\in\asX \
\vert \ \alpha_i\in\N_0, i\in\{0,\dots,n\}\right\rbrace.
\end{equation}
All definitions, notation, a=nd results from Sections \ref{sec:dVeronese} and
\ref{sec:Veronesemap} are valid but they can be rephrased in more explicit terms
replacing the abstract sets $\cN=\cN(I)$, $\cN_d$, and $\cN_{2d}$, respectively  with the explicit
set of ordered monomials $\cT=\cT(X)$, $\cT_d$, and $\cT_{2d}$.
In this case we consider the space $\textbf{k} \cT$
endowed with multiplication defined by
 $f \bullet g := \Nor_{\Re_{\cA}}(fg), \quad \text{for every}\; f,g \in \textbf{k}
 \cT.$
Then there is an isomorphism of graded algebras
\begin{equation}
\label{eq:Aas_vectspace1}
\cA=\cA(\textbf{k}, X_n, r) \cong (\textbf{k} \cT, \bullet ),
\end{equation}
and we identify the PBW algebra  $\cA$ with $(\textbf{k}\cT,   \bullet ).$
Similarly, the Yang-Baxter monoid
$S(X,r)$ is identified with $(\cT, \bullet)$.

We order the elements of $\cT_d$ lexicographically, so
        \begin{equation}
        \label{eq:deg-d-monomials1}
\cT_d = \{ w_1 =(x_1)^d < w_2= (x_1)^{d-1}x_2 < \dots < w_N= (x_n)^d \},
\;\text{where}\; N=\binom{n+d-1}{d}.
        \end{equation}
        The normalised $d$-Veronese solution, see Definition \ref{def:normalizedSol}, is denoted by $(\cT_d, \rho)=  (\cT_d, \rho_d)$.
The $d$-Veronese $\cA^{(d)}$ is a quadratic algebra
(one)-generated by $w_1, w_2, \dots , w_N.$

 It follows from \cite{PoPo}, Proposition 4.3, Ch 4,  that if $x_1, \cdots, x_n$ is a set of PBW generators of a quadratic algebra $A$, then
the elements of the PBW-basis of degree $d$ , taken in lexicographical order are PBW-generators of the Veronese subalgebra $A^{(d)}$.

\begin{cor}
In notation as above, the $d$-Veronese $\cA^{(d)}$ is a quadratic PBW algebra
with PBW generators the terms $w_1, w_2, \dots , w_N$  ordered lexicographically, see (\ref{eq:deg-d-monomials1}).
\end{cor}

For the class of finite square-free solutions $(X,r)$  Theorem \ref{thm:d-Veronese_relations} and especially the description of the set $\cR_b$
becomes more precise.
\begin{rmk}
\label{rmk:squarefree1}
If $w_i =x_{i_1}\cdots x_{i_d}, w_j =x_{j_1}\cdots
x_{j_d}\in \cT_{d}$, the product $w_iw_j$,   is
the leading monomial of an element $g_{ij}\in \cR_b$ if and only if $i_d > j_1$ and $\cdot\rho(w_i, w_j)\geq w_iw_j$.
\end{rmk}
 \begin{cor}
\label{cor:d-Veronese_relations_square-free}
Let $(X,r)$ be a finite
    square-free solution of order $n$,  let $X= \{x_1, \cdots, x_n\}$, be
    enumerated so that the algebra $\cA= \cA(\textbf{k}, X_n, r)$ is a binomial
    skew polynomial ring, let $d\geq 2$ be an integer, and $N=\binom{n+d-1}{d}$.
    Let  $(\cT_d, \rho)$ be the normalised $d$-Veronese solution.

The $d$-Veronese subalgebra $\cA^{(d)} \subseteq \cA$
is a quadratic PBW algebra
\[A^{(d)} \cong  \textbf{k} \langle w_1, \cdots, w_N \rangle/ (\cR),\]
with  PBW generators $\cT_{d}=\{w_1, \cdots, w_N\}$, and $N^2 -\binom{n+2d-1}{n-1}$ linearly independent quadratic relations $\cR$.
The relations $\cR$ split into two disjoint subsets $\cR= \cR_{a}
 \bigcup\cR_{b}$, described below.
 \begin{enumerate}
 \item
 \label{thm:d-Veronese_relations_Ra_SqFree}
The set $\cR_{a}$ contains $\binom{N}{2}$  relations corresponding
to the non-trivial $\rho$-orbits in $\cT_{d}\times \cT_{d}$:
 \begin{equation}
  \label{eq:fji_Sqfree}
\begin{array}{ll}
\cR_{a}=  \{f_{ji}= &w_jw_i - w_{i^{\prime}}w_{j^{\prime}}  \mid 1 \leq i,j \leq n,\; \text{where } \\
                &\rho(w_j, w_i)=(w_{i^{\prime}},w_{j^{\prime}}), \; \text{and }\;
                      w_j > w_{i^{\prime}} \; \text{holds in}\;\asX\}.
\end{array}
\end{equation}
Each monomial $w_iw_j$, such that $(w_i, w_j)$ is in a nontrivial
$\rho$-orbit occurs exactly once in  $\cR_{a}$. Every relation  $f_{ji}$ has leading monomial $\LM(f_{ji}) =
    w_jw_i$.
    \item
 \label{thm:d-Veronese_relations_Rb_SqFree}
The set $\cR_{b}$ contains $\binom{N+1}{2} -\binom{n+2d-1}{n-1}$
relations
 \begin{equation}
  \label{eq:gij_sqfree}
  \begin{array}{ll}
\cR_{b} =\{g_{ij} &= w_i w_j -w_{i_0}w_{j_0}\mid 1 \leq i,j \leq n,\; \text{where }\; w_i =x_{i_1}\cdots x_{i_d}, w_j =x_{j_1}\cdots
x_{j_d}\in \cT_{d}, \; i_d >j_1 \text{and}\\ & \cdot\rho(w_i, w_j)\geq w_iw_j, \;  w_{i_0}, w_{j_0}\in \cT_{d}\; \text{are such that}\;
 \Nor(w_iw_j) = w_{i_0}w_{j_0} \in \cT_{2d}\}.
\end{array}
\end{equation}
 In particular,
 $\LM(g_{ij}) = w_iw_j >  w_{i_0}w_{j_0}.$
\end{enumerate}
The relations $\cR$ form a Gr\"{o}bner basis of the ideal $(\cR)$ of the free associative algebra $\textbf{k}\langle w_1, \cdots, w_n\rangle$.
    \end{cor}

For a square-free solution $(X,r)$ and $(\cT_d, \rho_d)$ as above,
the $d$-Veronese solution $(Y, r_Y)$,  associated to  $(X,r)$,
is defined in Definition-Notation \ref{def:Yrho}. One has  $Y = \{y_1, y_2,
\cdots, y_N\}$,
and the map $r_Y:Y\times Y\longrightarrow Y\times Y$ is determined by
\begin{equation}
  \label{eq:rYsqfree}
r_Y (y_j, y_i) :=(y_{i^{\prime}},y_{j^{\prime}})\quad\text{\emph{iff}}\quad \rho
(w_j, w_i)= (w_{i^{\prime}},w_{j^{\prime}}), \; 1 \leq i,j,i^{\prime},
j^{\prime}\leq n.
\end{equation}
By definition $(Y, r_Y)$ is isomorphic to the solution $(\cT_d, \rho_d)$ .
Its Yang-Baxter algebra $\cA_Y = \cA(\textbf{k},Y,r_Y)$ is needed to define the $(n,d)$
Veronese homomorphism
$v_{n,d} : \cA_Y \rightarrow  \cA_X$
extending the assignment \[y_1 \mapsto w_1, \ y_2 \mapsto w_2, \ \dotsc ,\  y_N
\mapsto w_N.\]
Theorem \ref{thm:Veronese_ker} shows that the image of $v_{n,d}$ is the
$d$-Veronese subalgebra $\cA^{(d)}$ and determines a minimal set of generators
of its kernel.

The finite square-free solutions $(X,r)$ form an important subclass of the class
of all finite solutions, see for example \cite{GI12}. Moreover Theorem \ref{thm:PBWmain} shows that the Yang-Baxter algebra $\cA(\textbf{k},  X, r)$ of a finite solution $(X,r)$ is a PBW algebra if and only if $(X,r)$ is square-free.
So it is natural to ask \emph{"can we define analogue of
Veronese morphisms between Yang-Baxter algebras of
square-free solutions?"}
We shall prove that it is not
possible to restrict the defininition of Veronese maps introduced for
Yang-Baxter algebras of finite solutions to the subclass of Yang-Baxter algebras
of finte square-free solutions. Indeed,  if we assume that $(X,r)$ is square-free
then the algebra $A_Y$ involved in the definition of the map $v_{n,d}$ is
associated with the $d$-Veronese solution $(Y, r_Y)$,  which, in general is not
square-free, see Corollary  \ref{cor:Y_square-free}.

To prove the following result we work with the monomial $d$-Veronese solution $(S_d, r_d)$  keeping
in mind that it has special "hidden" properties induced by the braided monoid
$(S, r_S)$.
\begin{thm}
    \label{thm:Sd_square-free}
 Let $d\geq 2$ be an integer. Suppose $(X,r)$ is a finite square-free solution
 of order $n \geq 2$, $(S, r_S)$  is the associated braided monoid, and
$(S_d, r_d)$  is the monomial $d$-Veronese solution induced by $(S, r_S)$, see Def. \ref{def:VeroneseSol}.
Then $(S_d, r_d)$  is a square-free solution if and only if $(X,r)$ is a trivial
solution.
\end{thm}
\begin{proof}
Assume that $(S_d, r_d)$  is a square-free solution. We shall prove that $(X,r)$ is a
trivial solution.

Observe that if $(Z, r_Z)$ is a solution, then (i) $(Z, r_Z)$ is square-free if and
only if
\[{}^zz = z, \; \text{for all}\; z \in Z \]
and (ii) $(Z, r_Z)$ is the trivial solution
if and only if
\[
{}^yx = x, \; \text{for all}\; x,y \in Z.
 \]

Let $x,y \in X, x  \neq y$ and consider the monomial $a= x^{d-1}y\in S_d$. Our
assumption that $(S_d, r_d)$ is square-free implies that ${}^aa=a$ holds in
$S_d$, and therefore in $S$. Now Remark \ref{orbitsinG} implies
 the words $a$ and ${}^aa$ (considered as elements of $X^d$)  belong to the
 orbit $\Ocal=\Ocal_{\Dcal_m}(a)$ of $a=x^{d-1}y$ in  $X^d$.
We analyze  the orbit $\Ocal=\Ocal(x^{d-1}y)$ to find that it contains two type
of elements:
\begin{equation}\label{eq:u}
  u= ({}^{x^{d-1}}y)b, \; \text{where $b= {x^{d-1}}^y\in X^{d-1}$};
\end{equation}
and
\begin{equation}
\label{eq:u}
v = x^ic, \; \text{where $1 \leq i \leq d-1$ and $c\in  X^{d-i}$}.
\end{equation}
A reader who is familiar with the techniques and properties of square-free
solutions such as "cyclic conditions" and condition "lri"
may compute that $b ={(x^{d-1})}^y= (x^y)^{d-1}$ and $c = ({}^{x^{d-i-1}}y)(x^y)^{d-i-1}$, but
these details are not used in our proof.
We use condition ML2, see (\ref{eq:braided_monoid})
to yield the following equality in $S$:
\begin{equation}\label{eq:aa}
{}^aa={}^{(x^{d-1}y)}{(x^{d-1}y)}=({}^{x^{d-1}y}{x})({}^{(x^{d-1}y)^x}{x})\cdots
({}^{(x^{d-1}y)^{x^{d-1}}}{y})=\omega
\end{equation}
The assumption ${}^aa= a$ implies that word $\omega$, considered as an element of $X^d$ is in the orbit $\Ocal$ of $a$,
and therefore two cases are possible.

Case 1. The following is an equality of words in $X^d$:
\[
\omega=({}^{x^{d-1}y}{x})({}^{(x^{d-1}y)^x}{x})\cdots
({}^{(x^{d-1}y)^{x^{d-1}}}{y})=
({}^{x^{d-1}}y)b, \; \; b \in X^{d-1}.\]
Then there is an equality of elements of $X$:
\begin{equation}\label{eq:aa2}
{}^{(x^{d-1}y)}{x}={}^{x^{d-1}}y.
\end{equation}

Now we use condition ML1, see (\ref{eq:braided_monoid}) to obtain
\[
{}^{(x^{d-1}y)}{x}={}^{(x^{d-1})}{({}^yx)}
\]
which together with  (\ref{eq:aa2}) gives
\begin{equation}\label{eq:aa3}
{}^{(x^{d-1})}{({}^yx)}={}^{(x^{d-1})}y.
\end{equation}
The nondegeneracy implies that ${}^yx=y$. At the same time ${}^yy =y$, since
 $(X,r)$ is square-free, and using the nondegeneracy again one gets $x=y$, a
 contradiction.
It follows that Case 1 is impossible, whenever $x \neq y$.

Case 2. The following is an equality of words in $X^d:$
\[
\omega= ({}^{x^{d-1}y}{x})({}^{(x^{d-1}y)^x}{x})\cdots
({}^{(x^{d-1}y)^{x^{k-1}}}{y})=x^ic, \;\; \text{where} \;\; 1 \leq i\leq d-1,   c \in
X^{d-i}.
 \]
Then
\begin{equation}\label{eq:aa3}
{}^{(x^{d-1}y)}{x}=x.
\end{equation}
At the same time the equality ${}^xx=x$ and condition ML1 imply
${}^{x^{d-1}}x=x$,
which together with
(\ref {eq:aa3}) and ML1 (again) gives
\[
{}^{x^{d-1}}x = {}^{(x^{d-1}y)}{x}={}^{x^{d-1}}{({}^yx)}.
 \]
Thus, by the nondegeneracy ${}^yx = x$.
We have shown that ${}^yx = x$, for all   $x, y \in X, y \neq x$.
But $(X,r)$ is square-free, so ${}^yy = y$ for all  $y \in X$.
It follows that ${}^yx = x$ holds for all $x, y \in X$ and therefore $(X,r)$ is
the trivial solution.
\end{proof}

By construction the (abstract) $d$-Veronese solution $(Y,r_Y)$ associated to $(X,r)$ is isomorphic to the normalised solution $d$- Veroneze solution $(\cT_d, \rho_d)$
and therefore it is isomorphic to the monomial $d$-Veronese solution $(S_d,r_d)$. Theorem  \ref{thm:Sd_square-free} implies straightforwardly the following corollary.
\begin{cor}
    \label{cor:Y_square-free}
 Let $d\geq 2$ be an integer,
suppose $(X,r)$ is a square-free solution of finite order. Then the $d$-Veronese
solution $(Y, r_Y)$ is square-free if and only if $(X,r)$ is a trivial solution.
\end{cor}

\begin{rmk}
It follows from Corollary \ref{cor:Y_square-free} that the notion of Veronese morphisms introduced for the class of
Yang-Baxter algebras of finite solutions of YBE
can not be restricted to the subclass of Yang-Baxter algebras associated to finite square-free
solutions.
\end{rmk}

\subsection{Involutive permutation solutions}
\label{subsec:permutationsol}
Recall that a symmetric set $(X,r)$ is \emph{an involutive permutation solution} of Lyubashenko (or shortly \emph{a permutation solution}) if there exists a permutation $f \in \Sym(X)$, such that $r(x,y) =(f(y),f^{-1}(x))$. In this case we shall write $(X, f, r)$, see \cite{D}, and \cite{GI18}, p. 691.

\begin{pro}
\label{pro:permutationsolution}
Suppose $(X,f,r)$ is an involutive permutation solution of finite order $n$ defined as $r(x,y) =(f(y),f^{-1}(x))$, where $f$ is a permutation of $X$ and let $\cA$ be the associated Yang-Baxter algebra.
\begin{enumerate}
\item For every integer $d\geq 2$ the monomial $d$-Veronese solution $(S_d, r_d)$ is an involutive permutation solution.

\item If the permutation $f$ has order $m$ then for every integer $d$ divisible by $m$ the $d$-Veronese subalgebra $\cA^{(d)}$ of $\cA$ is a quotient of the commutative polynomial ring $\textbf{k}[y_1, y_2, \cdots, y_N ]$, where $N=\binom{n+d-1}{d}.$
    \end{enumerate}
\end{pro}
\begin{proof}

(1)
Let $q\geq 2$ be an integer.  The condition ML1 in (\ref{eq:braided_monoid}) implies that
\begin{equation}
\label{eq:perm1}
{}^at =f^{q}(t), \; \text{and}\; t^a =f^{-q}(t)=(f^{-1})^q(t),\quad \text{for all monomials}\; a \in S_q, \;\text{and all}\;\; t \in X.
\end{equation}
Moreover, since $S$ is a graded braided monoid  the monomials $a$, ${}^ba$ and $a^b$ have the same length, therefore

\begin{equation}
\label{eq:perm2}
{}^at = {}^{a^b}{t}=f^{q}(t), \quad t^a=t^{{}^ba}=f^{-q}(t),\; \text{for all }\; a \in S_q, \; b \in S, \;\text{and all}\;\; t \in X.
\end{equation}

It follows then from (\ref{eq:braided_monoid})  ML2  that $S$ acts on itself (on the left and on the right) as automorphisms. In particular, for $a,  t_1t_2\cdots t_d \in S_d$ one has
\begin{equation}
\label{eq:perm3}
\begin{array}{l}
{}^a{(t_1t_2\cdots t_d)}  =  ({}^a{t_1})({}^a{t_2})\cdots ({}^a{t_d}) = f^{d}(t_1) f^{d}(t_2) \cdots  f^{d}(t_d).\\
{(t_1t_2\cdots t_d)}^a  =  ({t_1}^a)({t_2}^a)\cdots ({t_d}^a) = f^{-d}(t_1) f^{-d}(t_2) \cdots  f^{-d}(t_d).
\end{array}
\end{equation}
Therefore $(S_d, r_d)$ is a permutation solution, $(S_d, f_d, r_d)$ where the permutation $f_d \in \Sym(S_d)$ is defined as $f_d(t_2t_2\cdots t_d) := f^{d}(t_1) f^{d}(t_2) \cdots  f^{d}(t_d)$. One has
$f_d^{-1}(t_2t_2\cdots t_d) := f^{-d}(t_1) f^{-d}(t_2) \cdots  f^{-d}(t_d)$.

(2) Assume now that $d=km$ for some integer $k \geq 1$, then $f^d = id_X$. It will be enough to prove that the monomial $d$-Veronese solution $(S_d, r_d)$ is the trivial solution.
  It follows from (\ref{eq:perm3})
that if $a \in S_d$ then
\begin{equation}
\label{eq:perm4}
{}^a{(t_1 t_2\cdots t_d)}  = t_1t_2\cdots t_d,\; \text{where}\;  t_i \in X,  1 \leq i \leq n.
\end{equation}

This implies ${}^ab= b$ for all $a, b \in S_d.$ Similarly, $a^b = a$ for all $a, b \in S_d.$
It follows that $(S_d, r_d)$ is the trivial solution. But the associated $d$-Veronese solution $(Y, r_Y)$ is isomorphic to $(S_d, r_d)$, hence $(Y, r_Y)$ is also a trivial solution, and therefore its Yang-Baxter algebra $\cA(\bf{k}, Y, r_Y)$ is the commutative polynomial ring $\textbf{k}[y_1, y_2, \cdots, y_N ]$.
It follows from Theorem \ref{thm:Veronese_ker} that the
$d$-Veronese subalgebra $\cA^{(d)}$ is isomorphic to the quotient $\textbf{k}[y_1, y_2, \cdots, y_N ]/(\mathfrak{K})$ where $\mathfrak{K}$ is the kernel of the Veronese map  $v_{n,d}$.
\end{proof}

\section{Examples}
\label{sec:examples}
We shall present two examples which illustrates the results of the paper. We use
the notation of the previous sections.

\begin{ex}[]\label{ex:n=3mod}
Let $n=3$, consider the solution $(X, r)$,  where
\[
\begin{array}{ll}
X= \{x_1, x_2, x_3\}, &\\
r(x_3, x_1) =  (x_2, x_3) & r(x_2, x_3) =  (x_3, x_1) \\
r(x_3, x_2) =  (x_1, x_3) & r(x_1, x_3) =  (x_3, x_2) \\
r(x_2, x_1) =  (x_1, x_2) & r(x_1, x_2) =  (x_2, x_1) \\
r(x_i, x_i) =  (x_i, x_i),& 1 \leq i \leq 3.
 \end{array}
 \]

Then
 \[
\begin{array}{l}
A(\textbf{k},X,r) = \textbf{k}\langle X  \rangle /(\Re_{\cA})
\;\;\text{where}\\
\Re_{\cA}
 = \{x_3x_2 -x_1x_3,\;  x_3x_1 -x_2x_3,\;   x_2x_1 -x_1x_2.
\}.
\end{array}
\]
The algebra $A=A(\textbf{k},X,r)$ is a PBW algebra with PBW generators $X=
\{x_1, x_2, x_3\}$, in fact it is a binomial skew-polynomial algebra.
 \end{ex}

 We first give an explicit presentation of the $2$-Veronese $\cA^{(2)}$ in terms
 of generators and quadratic relations.
 In this case $N=\binom{3+1}{2} =6$ and the $2$-Veronese subalgebra $\cA^{(2)}$
 is generated by $\cT_2$, the terms of length $2$ in $\textbf{k}\langle x_1,
 x_2,x_3\rangle.$ These are all normal (modulo $\Re_{\cA}$) monomials  of length $2$ ordered lexicographically:
\begin{equation}
 \label{eq:generatorsVeronese}
\begin{array}{l}
 \cT_2= \{w_1 =x_1x_1,\;  w_2 =x_1x_2 , \; w_3 =x_1x_3, \;w_4 =x_2x_2, \; w_5
 =x_2x_3,\;  w_6 =x_3x_3\}.
 \end{array}
 \end{equation}

Determine the normaized $2$-Veronese solution $(\cT_2, \rho_2)= (\cT_2, \rho)$, where $\rho(a,
b) = (\Nor ({}^ab), \Nor (a^b))$. An explicit description of $\rho$
is given below:
\begin{equation}
\label{eq:nontrivialorbits2}
\begin{array}{ll}
(x_3x_3,  w_i) \longleftrightarrow (w_i,  x_3x_3), & 1 \leq i \leq  5 \\
(x_2x_3, x_2x_3) \longleftrightarrow (x_1x_3, x_1x_3), & (x_2x_3, x_2x_2)
\longleftrightarrow (x_1x_1, x_2x_3),\\
(x_2x_3, x_1x_2)\longleftrightarrow (x_1x_2, x_2x_3),  &   (x_2x_3, x_1x_1)
\longleftrightarrow (x_2x_2, x_2x_3), \\
(x_2x_2, x_1x_3) \longleftrightarrow (x_1x_3, x_1x_1),  &   (x_2x_2,
x_1x_2)\longleftrightarrow (x_1x_2, x_2x_2),\\
 (x_2x_2, x_1x_1)\longleftrightarrow (x_1x_1, x_2x_2),  &(x_1x_3,
 x_2x_2)\longleftrightarrow (x_1x_1, x_1x_3), \\
 (x_1x_3, x_1x_2)\longleftrightarrow (x_1x_2, x_1x_3),  &  (x_1x_2, x_1x_1)
 \longleftrightarrow  (x_1x_1, x_1x_2).\\
 \end{array}
 \end{equation}

 \emph{The fixed points} $ \Fcal= \Fcal (\cT_2, \rho_2)$  are the monomials $ab$
 determined by the one-element orbits of $\rho$, one has $(a,b) = ({}^ab, a^b)$.
 There are exactly $6$ fixed points:
\begin{equation}
 \label{eq:fixedpts11}
\begin{array}{ll}
 \Fcal=& \{w_1w_1 =(x_1x_1)(x_1x_1)\in \cT_4, \; w_4w_4 =(x_2x_2)(x_2x_2)\in
 \cT_4,\; w_6w_6 =(x_3x_3) (x_3x_3) \in \cT_4,\\
           &w_2w_2 =(x_1x_2)(x_1x_2)\notin \cT_4, \;w_3w_5
           =(x_1x_3)(x_2x_3)\notin \cT_4,\; w_5w_3 =(x_2x_3)(x_1x_3)\notin \cT_4.
           \}.
 \end{array}
 \end{equation}

There are exactly $15= \binom{N}{2}$ nontrivial $\rho$-orbits in $\cT_2 \times
\cT_2$ determined by  (\ref{eq:nontrivialorbits2}).
 These orbits imply the following equalities in
 $\cA^{(2)}$:

 \begin{equation}
 \label{eq:relations11}
\begin{array}{ll}
(x_3x_3)w_i=w_i (x_3x_3) \in \cT_4,1 \leq i \leq  5, & \\
(x_2x_3) (x_2x_3) = (x_1x_3)(x_1x_3)\notin \cT_4, &(x_2x_3)(x_2x_2) =
(x_1x_1)(x_2x_3)\in \cT_4,\\
(x_2x_3) (x_1x_2)= (x_1x_2)(x_2x_3)\in \cT_4, & (x_2x_3)( x_1x_1) = (x_2x_2,
x_2x_3)\in \cT_4,\\
 (x_2x_2)(x_1x_3) = (x_1x_3)( x_1x_1)\notin \cT_4, &(x_2x_2)(x_1x_2)=
 (x_1x_2)(x_2x_2)\in \cT_4,\\
(x_2x_2) (x_1x_1)= (x_1x_1) (x_2x_2)\in \cT_4,      & (x_1x_3)(
x_2x_2)=(x_1x_1)(x_1x_3)\in \cT_4,\\
 (x_1x_3)(x_1x_2)=(x_1x_2)(x_1x_3)\notin \cT_4, &(x_1x_2)(x_1x_1) =
 (x_1x_1)(x_1x_2)\in \cT_4.
 \end{array}
 \end{equation}
Note that for every pair $(w_i,w_j) \in \cT_2\times \cT_2 \setminus \Fcal$ the
monomial $w_iw_j$ occurs exactly once in ( \ref{eq:relations11}) .

 Six additional quadratic relations of $\cA^{(2)}$ arise  from
 (\ref{eq:fixedpts11}), (\ref{eq:relations11}), and the obvious equality
$a = \Nor(a)\in \cT$, which hold in $\cA^{(2)}$ for every $a \in X^2$. In
this case we simply pick up all monomials which occur in (\ref{eq:fixedpts11}),
or (\ref{eq:relations11})  but are not in $\cT_4$ and equalize each of them with its normal form. This way we get the six relations which determine
$\cR_{b}$:
 \begin{equation}
 \label{eq:relations2}
\begin{array}{lll}
 (x_1x_2)(x_1x_2) = (x_1x_1)(x_2x_2), & (x_1x_3)(x_2x_3)=(x_1x_1)(x_3x_3),
 &(x_2x_3)(x_1x_3)=(x_2x_2)(x_3x_3)\\
(x_1x_3)(x_1x_3)= (x_1x_2)(x_3x_3),   &(x_2x_2)(x_1x_3) = (x_1x_2)(x_2x_3),
&(x_1x_2)(x_1x_3)= (x_1x_1)(x_2x_3).
 \end{array}
 \end{equation}

The 2-Veronese algebra $\cA^{(2)}$ has 6 generators $w_1, \cdots, w_6$ written
explicitly in  (\ref{eq:generatorsVeronese})
and  a set of 21 relations presented as
a disjoint union $\cR= \cR_{a} \bigcup\cR_{b}$ described below.

(1) The relations $\cR_{a}$ are:
\begin{equation}
 \label{eq:relationsRa}
\begin{array}{ll}
w_6w_i-w_i w_6,  \;w_i w_6\in \cT_4,\; 1 \leq i \leq  5, &  \\
w_5w_5- w_3w_3,  \; w_3w_3\notin \cT_4, &w_5w_4-w_1w_5,  \;w_1w_5\in \cT_4,\\
w_5w_2 -w_2w_5,  \; w_2w_5 \in \cT_4,  &w_5w_1-w_4w_5,  \;w_4w_5\in \cT_4, \\
 w_4w_3 -w_3w_1,  \; w_3w_1  \notin \cT_4,  & w_4w_2-w_2w_4,  \;w_2w_4 \in
 \cT_4,\\
 w_4w_1- w_1w_4,  \; w_1w_4\in \cT_4,  &w_3w_4 -w_1w_3,  \; w_1w_3  \in \cT_4,
 \\
 w_3w_2 - w_2w_3,  \; w_2w_3\notin \cT_4, &  w_2w_1 -w_1w_2,  \;w_1w_2 \in
 \cT_4.\\
 \end{array}
 \end{equation}

(2) The relations $\cR_{b}$ are:
\begin{equation}
 \label{eq:relationsRb}
\begin{array}{lll}
 w_2w_2- w_1w_4,\quad &w_3w_5 - w_1w_6,\quad& w_5w_3- w_4w_6, \\
 w_3w_3  - w_2w_6 ,        &w_3w_1 - w_2w_5,          &w_2w_3-   w_1w_5.
 \end{array}
 \end{equation}
The elements of $\cR_{b}$ correspond to the generators of the kernel of the
Veronese map.

(1a) The relations $\cR_{a1}$ are:
\begin{equation}
 \label{eq:relationsRa1}
\begin{array}{ll}
w_6w_i-w_i w_6, \; w_i w_6 \in \cT_4,\; 1 \leq i \leq  5, &  \\
w_5w_5- w_2w_6, \;  w_2w_6  \in \cT_4, &w_5w_4-w_1w_5, \; w_1w_5\in \cT_4,\\
w_5w_2 -w_2w_5, \; w_2w_5 \in \cT_4,  &w_5w_1-w_4w_5, \; w_4w_5\in \cT_4, \\
 w_4w_3 -  w_2w_5, \: w_2w_5\in \cT_4,  & w_4w_2-w_2w_4, \; \in \cT_4,\\
 w_4w_1- w_1w_4, \; w_1w_4\in \cT_4,  &w_3w_4 -w_1w_3, \;  w_1w_3\in \cT_4, \\
 w_3w_2 - w_1w_5, \; w_1w_5\in \cT_4, &  w_2w_1 =w_1w_2, \; w_1w_2\in \cT_4.\\
 \end{array}
 \end{equation}

Thus the $2$-Veronese $\cA^{(2)}$ of the algebra $\cA$  is a quadratic algebra
presented as
\[ \cA^{(2)} \simeq  \textbf{k}\langle w_1, \cdots, w_6 \rangle /(\cR) \simeq
\textbf{k}\langle w_1, \cdots, w_6 \rangle /(\cR_1),
\]
where , $\cR = \cR_{a}\cup \cR_{b}$, and $\cR_1 = \cR_{a1}\cup \cR_{b}$.

\emph{The $2$-Veronese subalgebra $\cA^{(2)}$ in our example is a PBW algebra}.

The associated $2$-Veronese solution of YBE $(Y, r_Y)$ can be found
straightforwardly: one has $Y = \{y_1, y_2, y_3, y_4, y_5, y_6\}$ and $r_Y(y_i,
y_j) = (y_k, y_l)$ \emph{iff}  $\rho_2(w_i, w_j) = (w_k, w_l)$, $1 \leq i, j, k,
l \leq 6$. The solution $(Y, r_Y)$ is nondegenerate and involutiove, but it is
not a square-free solution.
The corresponding Yang-Baxter algebra is
\[\cA_Y = (\textbf{k}, Y, r_Y)= \textbf{k}\langle y_1, y_2, y_3, y_4, y_5,
y_6\rangle /(\Re), \]
where $\Re$ is the set of quadratic relations given below:
\begin{equation}
 \label{eq:relationsAY}
\begin{array}{lll}
y_6y_i-y_i y_6,& 1 \leq i \leq  5, &y_5y_5- y_3y_3,  \\
y_5y_4-y_1y_5, &y_5y_1-y_4y_5,     & y_5y_2 -y_2y_5  \\
 y_4y_3 -y_3y_1, &y_4y_2-y_2y_4, &y_4y_1- y_1y_4,\\
 y_3y_4 -y_1y_3, &y_3y_2 - y_2y_3, &y_2y_1 =y_1y_2.\\
 \end{array}
 \end{equation}
\emph{Note that $\Re$ is not a Gr\"{o}bner basis of the ideal $(\Re)$ (w.r.t.
the degree-lexicographic ordering on $\langle Y \rangle$)}. For example,
the overlap $y_5y_4y_3$ implies the new relation $y_5y_3y_1-y_1y_5y_3$ which is
in the ideal $(\Re)$ but can not be reduced using (\ref{eq:relationsAY}). There
are more such overlaps.

 The Veronese map,
\[v_{n,2} : \cA_Y \rightarrow  \cA_X\]
is the algebra homomorphism extending the assignment $y_1 \mapsto w_1, \ y_2
\mapsto w_2, \ \dotsc ,\  y_6  \mapsto w_6.$ Its image is the $2$-Veronese, $\cA^{(2)}$.
The kernel $\mathfrak{K}$ of the map  $v_{n,2}$ is generated by the set
$\Re_1$   of polynomials given below:
\begin{equation}
 \label{eq:kernelveronese_ex}
\begin{array}{lll}
 y_2y_2-y_1y_4, &y_3y_5 - y_1y_6, &y_5y_3-y_4y_6\\
y_3y_3-y_2y_6,  &y_4y_3 -y_2y_5, &y_2y_3-y_1y_5.
 \end{array}
 \end{equation}

Denote by $J$ the two-sided ideal  $J= (\Re\cup \Re_1)$ of $ \textbf{k}\langle Y
\rangle$. A direct computation shows that the set
$\Re \cup \Re_1$ is a (quadratic) Gr\"{o}bner basis of $J$. The $2$-
Veronese subalgebra $\cA^{(2)}$ of $\cA=\cA_X$ is isomorphic to
the quotient $\textbf{k}\langle Y \rangle/ J,$ hence it is a PBW algebra.

\begin{ex}[]\label{ex:n=2}
Let $n=2$, consider the solution $(X, r)$,  where
\[
\begin{array}{ll}
X= \{x_1, x_2\}, &\\
r(x_2, x_2) =  (x_1, x_1) & r(x_1, x_1) =  (x_2, x_2) \\
r(x_2, x_1) =  (x_2, x_1) & r(x_1, x_2) =  (x_1, x_2) .
\end{array}
 \]
This is a permutation solution $(X, f, r)$, where $f$ is the transposition $f=(x_1 x_2)$.
One has
 \[
\begin{array}{l}
A(\textbf{k},X,r) = \textbf{k}\langle x_1, x_2 \rangle /(\Re_{\cA})
\;\;\text{where}\\
\Re_{\cA}
 = \{x_2x_2 -x_1x_1\}.
\end{array}
\]
\end{ex}
The set $\Re_{\cA}$ is not a Gr\"{o}bner basis of $I= (\Re_{\cA})$ with respect to the deg-lex ordering induced by any
of the choices $x_1 < x_2,$ or $x_2< x_1$.
We keep the convention $x_1 < x_2$ and apply standard computation to find that the reduced
Gr\"{o}bner basis  of $I$ (with respect to the deg-lex ordering) is
\[G = \{f_1 = x_2x_2 -x_1x_1, f_2 = x_2x_1x_1 -x_1x_1 x_2\}.\]
Then
\[
\cN = \cN(I)= \{x_1^{\alpha} (x_2x_1)^{\beta}x_2^{\varepsilon}\mid \varepsilon
\in\{0, 1\} \; \text{and}\;  \alpha, \beta \in  \N_0   \}.
\]

It is easy to find an explicit presentation of the $2$-Veronese $\cA^{(2)}$ in terms
of generators and quadratic relations.
 $\cA^{(2)}$  is generated by the set of normal monomials of length $2$:
 \[\cN_2 = \{w_1 =x_1x_1, w_2= x_1x_2, w_3= x_2x_1\}.\]

One has
\[
\cN_4 = \{x_1^{4}, \; x_1^3x_2,\;  x_1^2x_2x_1, \; x_1x_2x_1x_2, \;x_2x_1x_2x_1
\}.
\]

Next we determine the normalised $d$-Veronese solution $(\cN_2, \rho)$, where
$\rho(a, b) = (\Nor ({}^ab), \Nor (a^b))$.
One has
\[{}^{(x_ix_j)}x_k = x_k, \quad x_k^{(x_ix_j)} = x_k,\quad \text{for all}\; i, j,
k \in \{1, 2\} \]
\[{}^{w_i}{w_j}=w_j, \; w_j^{w_i}=w_j, \; \forall i, j \in \{1, 2, 3\}.\]

Thus $(\cN_2, \rho)$ is the trivial solution on the set $\cN_2$:
\[\rho(w_j, w_i) = (w_i, w_j), 1 \leq i,j \leq 3.\]
In this case the three  fixed points are normal monomials:
 \[
\begin{array}{ll}
 \Fcal=& \{w_1w_1 =(x_1x_1)(x_1x_1)\in \cN_4, \; w_2w_2 =(x_1x_2)(x_1x_2)\in
 \cN_4,\; w_3w_3 =(x_2x_1) (x_2x_1) \in \cN_4  \}.
 \end{array}
 \]

The set of relations is $\cR = \cR_a \cup \cR_b$.
Here $\cR_a$ consists of the relations:
\begin{equation}
\begin{array}{lcl}
g_{32}= w_3w_2-w_2w_3,   & \text{equivalently}, &
(x_2x_1)(x_1x_2)=(x_1x_2)(x_2x_1) \notin \cN_4\\
g_{31}= w_3w_1-w_1w_3,  &\text{equivalently}, & (x_2x_1)(x_1x_1)=(x_1x_1)(x_2x_1)
\in \cN_4\\
g_{21}  = w_2w_1-w_1w_2,  & \text{equivalently},&
(x_1x_2)(x_1x_1)=(x_1x_1)(x_1x_2) \in \cN_4.
\end{array}
\end{equation}
There is only one relation in $\cR_b$, it gives the "normalisation" of
$w_2w_3=(x_1x_2)(x_2x_1) \notin \cN_4$.
One has
\[
\Nor (w_2w_3)=\Nor (x_1x_2)(x_2x_1) = \Nor (x_1(x_2x_2)x_1) = x_1x_1x_1x_1
=w_1w_1,
\]
hence
\begin{equation}
\label{eq:R_b2}
\cR_b=\{g_{23}= w_2w_3 - w_1w_1\}.
\end{equation}
It follows that
\[
\begin{array}{l}
\cA^{(2)} \simeq  \textbf{k}\langle w_1, w_2, w_3 \rangle /(\cR)
\;\text{where}\\
\cR = \{w_3w_2-w_2w_3, \; w_3w_1 - w_1w_3,\; w_2w_1-w_1w_2,\; w_2w_3 - w_1w_1 \}.
\end{array}
\]
In our notation the second set $\cR_1$ consisting of equivalent relation is:
\[\cR_1 = \{w_3w_2-w_1w_1, \; w_3w_1 - w_1w_3,\; w_2w_1-w_1w_2,\; w_2w_3 -
w_1w_1\},\]
and $
\cA^{(2)} \simeq  \textbf{k}\langle w_1, w_2, w_3 \rangle /(\cR_1).
$
It is easy to see that that the set $\cR$ is a (minimal) Gr\"{o}bner basis of the two
sided ideal $I = (\cR)$ of $\textbf{k}\langle w_1, w_2 w_3 \rangle$,
 w.r.t. the degree-lexicographic order on $\langle w_1, w_2, w_3\rangle$,
while the set $\cR_1$ is the reduced Gr\"{o}bner basis  of the ideal $I$.
Thus the $2$-Veronese subalgebra $\cA^{(2)}$ in this example is
a PBW algebra.
As expected, the $2$-Veronese $\cA^{(2)}$ is a commutative algebra isomorphic to $
\textbf{k}[w_1, w_2, w_3]/( w_2w_3 - w_1w_1)$.

The associated $2$-Veronese solution of YBE $(Y, r_Y)$ is the trivial solution
on the set $Y = \{y_1, y_2, y_3\}$,
$r_Y(y_j, y_i) = (y_i, y_j)$ for all $1 \leq i, j \leq 3$.
The corresponding Yang-Baxter algebra $\cA_Y$ is
\[\cA_Y = (\textbf{k}, Y, r_Y)= \textbf{k}\langle y_1, y_2, y_3\rangle
/(y_3y_2-y_2y_3,  y_3y_1-y_1y_3, y_2y_1-y_1y_2) \simeq \textbf{k}[ y_1, y_2,
y_3].\]
Obviously, $\cA_Y$ is PBW.
The Veronese map,
\[v_{n,2} : \cA_Y \rightarrow  \cA_X\]
is the algebra homomorphism extending the assignment $y_1 \mapsto w_1, \ y_2
\mapsto w_2, y_3  \mapsto w_3.$
Its image is the $d$-Veronese $\cA^{(2)}$, and its kernel $\mathfrak{K}$ is generated by the
polynomial $y_2y_3 - y_1y_1.$


\begin{thebibliography}{99}

\bibitem{Anickmon}
D.~ Anick,  {\it On monomial algebras of finite global dimension}. Trans. AMS   {\bf 291}  (1985) 291-310.

\bibitem{AS}
M.~ Artin\ and\ W. F.~ Schelter, {\it Graded algebras of global dimension $3$},
Adv. in Math. {\bf 66} (1987), no.~2, 171--216. \MR{0917738}


\bibitem{AGG}  Arici, Francesca, Francesco Galuppi, \ and\ Tatiana
    Gateva-Ivanova, {\it Veronese and Segre morphisms between non-commutative
    projective spaces},  to appear in European Journal of Mathematics (2022),
arXiv preprint arXiv:2003.01681 (2020).


\bibitem{Backelin}
J. ~Backelin,  {\it On the rates of growth of the homologies of Veronese
subrings},  Algebra, algebraic topology and their interactions, Springer,
Berlin, Heidelberg, 1986. 79--100.



\bibitem{Bergman}
G. M. Bergman, {\it The diamond lemma for ring theory}, Adv. in Math. {\bf 29}
(1978), no.~2, 178--218. \MR{0506890}


\bibitem{D}
V.~G. Drinfeld, {\em On some unsolved problems in quantum group
theory}, Quantum Groups (P.~P. ~Kulish, ed.), Lect. Notes in Mathematics, vol.
1510,
  Springer Verlag, 1992, ~1--8.


\bibitem{ESS}
P. ~Etingof,  T. ~Schedler, \ and A. ~Soloviev,
{\em{Set-theoretical solutions to the quantum
Yang--Baxter equation}},
Duke Math.\ J.\ {\bf{100}}  (1999), 169--209.

\bibitem{FRT}
L. D. Faddeev, N. Yu. Reshetikhin\ and\ L. A. Takhtajan, Quantization of Lie groups and Lie algebras, in {\it Algebraic analysis, Vol. I}, 129--139, (1989) Academic Press, Boston, MA. \MR{0992450}


\bibitem{Froberg}
 R.~Fr\"{o}berg, and J. ~Backelin, {\em{Koszul algebras, Veronese subrings, and
 rings with linear resolutions}},
Rev. Roumaine Math. Pures Appl., 30 (1985), 85--97



 \bibitem{GI94}  T. Gateva-Ivanova, {\em{Noetherian properties of skew
     polynomial rings with binomial relations}}, Trans. Amer. Math. Soc. {\bf
     343} (1994), no.~1, 203--219.
     \MR{1173854}
\bibitem{GI96}
T.~Gateva--Ivanova, {\em{Skew polynomial rings with binomial
relations}}, J.\ Algebra {\bf{185}} (1996) ~710--753.


\bibitem{GI04}
T.~Gateva--Ivanova,
{\em{A combinatorial approach to the set-theoretic
solutions of the Yang--Baxter equation}},
 J.Math.Phys., {\bf{45}} (2004) ~3828--3858.


\bibitem{GI04s}
T.~Gateva--Ivanova, {\em Quantum binomial algebras,
Artin-Schelter regular rings, and solutions of the Yang--Baxter
equations}, Serdica Math. J. {\bf{30}}  (2004) ~431--470.



\bibitem{GI11}
T. Gateva-Ivanova, {\em{Garside structures on
monoids with quadratic square-free relations}} , Algebr. Represent.
Theor \textbf{14} (2011), ~779--802.


\bibitem{GI12}
T.~Gateva-Ivanova,
{\em{Quadratic algebras, Yang-Baxter equation, and Artin--
Schelter regularity}}, Advances in Mathematics, \textbf{230}  (2012),
~ 2152--2175.


\bibitem{GI18}
 T. Gateva--Ivanova, {\em{Set--theoretic solutions of the
Yang--Baxter equation, Braces and Symmetric groups}},
 Advances in Mathematics \textbf{338} (2018), 649--701.

\bibitem{GI21}
  T.~Gateva--Ivanova,
  {\em{A combinatorial approach to noninvolutive set-theoretic solutions of the
  Yang-Baxter equation}},
Publicacions Matem\`{a}tiques 65.2 (2021): 747--808.

\bibitem{GI23}
 T.~Gateva--Ivanova,
 {\em{Segre products and Segre morphisms in a class of Yang–Baxter algebras}}, Letters in Mathematical Physics 113, 34 (2023). https://doi.org/10.1007/s11005-023-01657-z



\bibitem{GIM08}
T.~ Gateva-Ivanova, S.~Majid, {\em{Matched pairs approach
to set theoretic solutions of the Yang--Baxter equation}}, J.
Algebra {\bf{319}} (2008) ~ 1462--1529.

\bibitem{GIM11}
T. ~Gateva-Ivanova, S. ~Majid, {\em{Quantum
spaces associated to multipermutation solutions of
   level two}},   Algebr. Represent. Theor., \textbf{14} (2011) 341--376



\bibitem{GIVB}
T.~Gateva-Ivanova,  M.~Van den Bergh, {\em{Semigroups of
$I$-type}}, J.\ Algebra {\bf{206}} (1998) ~97--112.








 \bibitem{Harris} J. Harris, {\it Algebraic geometry}, Graduate Texts in
     Mathematics, 133, Springer-Verlag, New York, 1992. \MR{1182558}


 \bibitem{KRW90}
 A. Kandri-Rody\ and\ V. Weispfenning, Noncommutative Gr\"{o}bner bases in
 algebras of solvable type, J. Symbolic Comput. {\bf 9} (1990), no.~1, 1--26.
 \MR{1044911}


\bibitem{Latyshev}
V.N. Latyshev, Combinatorial ring theory. Standard bases,  Izd. Mosk. Univ.,
Moscow (1988).


\bibitem{MajidQG}
S.~Majid, {\em{ Foundations of the Quantum Groups}},
Cambridge University Press, 1995.




\bibitem{Ma88}
Yu. I. Manin, {\it Quantum groups and noncommutative geometry}, Universit\'{e}
de Montr\'{e}al, Centre de Recherches Math\'{e}matiques, Montreal, QC, 1988.
\MR{1016381}


\bibitem{Ma91}
Y. I. Manin, {\it Topics in noncommutative geometry}, M. B. Porter Lectures, Princeton University Press, Princeton, NJ, 1991. \MR{1095783}

\bibitem{Mo88}
T. Mora, Groebner bases in noncommutative algebras, in {\it Symbolic and
algebraic computation (Rome, 1988)}, 150--161, Lecture Notes in Comput. Sci.,
358, Springer, Berlin. \MR{1034727}


 \bibitem{Mo94}
T. Mora, {\it An introduction to commutative and noncommutative Gr\"{o}bner bases},
Theoret. Comput. Sci. {\bf 134} (1994), no.~1, 131--173. \MR{1299371}






 \bibitem{PoPo}
 A. Polishchuk\ and\ L. Positselski, {\it Quadratic algebras}, University
 Lecture Series, 37, American Mathematical Society, Providence, RI, 2005.
 \MR{2177131}

\bibitem{priddy}
S. B. Priddy, {\it Koszul resolutions}, Trans. Amer. Math. Soc. {\bf 152} (1970),
39--60. \MR{0265437}

\bibitem{RTF}
N.Yu.~ Reshetikhin, L.A.~ Takhtadzhyan, L.D. ~Faddeev,
{\em{Quantization of Lie groups and Lie algebras}}, Algebra i
Analiz 1 (1989) 178-206 (in Russian); English translation in:
Leningrad Math. J. 1 (1990) 193-225.

\bibitem{rump}
W.~ Rump,  {\em{A decomposition theorem for square-free unitary solutions of the quantum Yang-Baxter equation,}}  Advances in Mathematics 193, no. 1 (2005): 40-55.

 \bibitem{Sm03}
S. P. ~ Smith, {\em{ Maps between non-commutative spaces}}, Trans. Amer. Math. Soc. {\bf
356} (2004), no.~7, 2927--2944. \MR{2052602}

\bibitem{vendramin}
L.~ Vendramin, {\em{ Extensions of set-theoretic solutions of the Yang–Baxter equation and a conjecture of Gateva-Ivanova}} Journal of Pure and Applied Algebra 220.5 (2016): 2064-2076.
\end{thebibliography}
\end{document}